\documentclass[a4paper,11pt]{article}
\usepackage[T1]{fontenc}
\usepackage{lmodern,amsmath,amsthm,amsfonts,amssymb,graphicx,float,microtype,thmtools,underscore,mathtools,xurl,thm-restate}
\usepackage[dvipsnames,svgnames,table]{xcolor}
\usepackage[shortlabels]{enumitem}
\setlist[itemize]{topsep=0ex,itemsep=0ex,parsep=0ex}
\setlist[enumerate]{topsep=0ex,itemsep=0ex,parsep=0ex}
\usepackage{pifont}

\usepackage[unicode=true]{hyperref}
\hypersetup{
colorlinks,
breaklinks=true,
linkcolor={blue!60!black},
citecolor={black},
urlcolor={blue!60!black},
pdftitle={Countable Graphs with Finite Path-width: Characterisation and Universality}}
\usepackage[capitalise, compress, nameinlink, noabbrev]{cleveref}
\crefname{lem}{Lemma}{Lemmas}
\crefname{thm}{Theorem}{Theorems}
\crefname{cor}{Corollary}{Corollaries}
\crefname{claim}{Claim}{Claims}
\newcommand{\defn}[1]{\textcolor{Maroon}{\emph{#1}}}
\newcommand{\mathdefn}[1]{\textcolor{Maroon}{#1}}
\usepackage[longnamesfirst,numbers,sort&compress]{natbib}
\makeatletter
\def\NAT@spacechar{~}
\makeatother
\usepackage[tmargin=30mm,bmargin=30mm,lmargin=30mm,rmargin=30mm]{geometry}
\renewcommand{\baselinestretch}{1.1}
\allowdisplaybreaks
\DeclareMathOperator{\pw}{pw}
\DeclareMathOperator{\lw}{lw}
\DeclareMathOperator{\wow}{wow}
\DeclareMathOperator{\dist}{dist}

\DeclarePairedDelimiter{\set}{\lbrace}{\rbrace} 
\renewcommand{\epsilon}{\varepsilon}
\renewcommand{\emptyset}{\varnothing}

\renewcommand{\geq}{\geqslant}
\renewcommand{\leq}{\leqslant}

\newcommand{\GG}{\mathcal{G}}

\newcommand{\NN}{\mathbb{N}}

\newcommand{\C}{\mathcal{C}}
\renewcommand{\thefootnote}{\fnsymbol{footnote}}
\theoremstyle{plain}
\newtheorem{thm}{Theorem}
\newtheorem{lem}[thm]{Lemma}
\newtheorem{cor}[thm]{Corollary}

\newtheorem{prop}[thm]{Proposition}
\newtheorem{claim}{Claim}
\crefname{obs}{Observation}{Observations}
\newtheorem*{lem*}{Lemma}
\theoremstyle{definition}

\newtheorem*{conj*}{Conjecture}

\begin{document}
\title{\bf\boldmath\fontsize{16pt}{18pt}\selectfont Countable Graphs with Finite Path-width: Characterisation and Universality}

\author{%
\qquad Tony Huynh\,\footnotemark[1]
\qquad Freddie Illingworth\,\footnotemark[2]
\qquad Nikolai Karol\,\footnotemark[3]
\\
Florian Lehner\,\footnotemark[4]
\qquad Chun-Hung Liu\,\footnotemark[5]
\qquad J\'anos Pach\,\footnotemark[6]
\qquad David~R.~Wood\,\footnotemark[3]
}

\maketitle

\begin{abstract}
We study path-width and the closely related parameter line-width in countably infinite graphs. Our first result characterises the graphs of finite path-width: they are the graphs that do not have infinitely many vertices of infinite degree, do not have infinitely many pairwise disjoint infinite paths, and contain no subdivision of some finite tree of maximum degree 3. We then investigate universality under the subgraph relation for graphs of bounded path-width or line-width. In particular, we prove that there exists a universal graph with line-width $\mathcal{O}(k^2)$ for the class of graphs with line-width at most $k$. In contrast, we show that no graph of finite path-width is universal for the class of locally finite graphs with path-width $1$. Finally, we show that for each $k\geq 2$, every universal graph for the class of graphs with path-width at most $k$ has line-width at least $k + 1$.
\end{abstract}

\footnotetext[1]{Discrete Mathematics Group, Institute for Basic Science (IBS), Daejeon, South Korea (\href{mailto:tony@ibs.re.kr}{tony@ibs.re.kr}).  Supported by the Institute for Basic Science (IBS-R029-C1).}
\footnotetext[2]{Department of Mathematics, University College London, UK (\textsf{\href{mailto:f.illingworth@ucl.ac.uk}{f.illingworth@ucl.ac.uk}}).}
\footnotetext[3]{School of Mathematics, Monash University, Melbourne, Australia (\textsf{\{\href{mailto:nikolai.karol@monash.edu}{nikolai.karol},\href{mailto:david.wood@monash.edu}{david.wood}\}\allowbreak@monash\allowbreak.edu}). Research of Wood is supported by the Australian Research Council and by NSERC.}
\footnotetext[4]{Department of Mathematics, University of Auckland, New Zealand (\textsf{\href{mailto:florian.lehner@auckland.ac.nz}{florian.lehner@auckland.ac.nz}}).}
\footnotetext[5]{Department of Mathematics, Texas A\&M University, USA (\href{mailto:chliu@tamu.edu}{chliu@tamu.edu}). Partially supported by NSF under CAREER award DMS-2144042.}
\footnotetext[6]{Alfr\'ed R\'enyi Institute of Mathematics, Budapest, Hungary (\href{mailto:pach@renyi.hu}{pach@renyi.hu}). Supported by the ERC Advanced Grant `GeoScape.'}


\renewcommand{\thefootnote}{\arabic{footnote}}
\section{Introduction}
\label{sec:intro}

Path-width measures how similar a given graph is to a path and is of fundamental importance in structural and algorithmic graph theory; see \citep{Bodlaender98} for a survey. This paper studies the path-width of infinite graphs\footnote{In this paper, all graphs are simple and countable; that is, the vertex set is finite or countably infinite. A graph $U$ \defn{contains} a graph $G$ if a subgraph of $U$ is isomorphic to $G$. A \defn{class} is a collection of graphs, closed under isomorphism. A class $\GG$ is \defn{monotone} if for every graph $G\in\GG$ every subgraph of $G$  is in $\GG$.}; in particular, we (1) characterise the infinite graphs with finite path-width, and (2) study universality for the class of graphs with given path-width and other closely related parameters. 

A \defn{path-decomposition} of a graph $G$ is a sequence $(B_t : t \in \mathbb{N})$ of subsets of $V(G)$ (called \defn{bags}) such that: 
\begin{itemize}
    \item $\bigcup_{t \in \mathbb{N}}B_{t} = V(G)$,
    \item for each $vw \in E(G)$, there exists $t \in \mathbb{N}$ such that $v, w \in B_{t}$, and
    \item $B_{t_{1}} \cap B_{t_{3}} \subseteq B_{t_{2}}$ for  all $t_{1}, t_{2}, t_{3} \in \mathbb{N}$ with $t_{1} \leqslant t_{2} \leqslant t_{3}$; equivalently, for every $v \in V(G)$, the bags containing $v$ form a non-empty consecutive subsequence of $(B_1, B_{2}, \dots)$.
\end{itemize}

The \defn{width} of a path-decomposition $(B_t : t \in \mathbb{N})$ is $\sup_{t \in \mathbb{N}}(|B_t|-1)$.  The \defn{path-width} of a graph $G$, denoted by \defn{$\pw(G)$}, is the minimum width of a path-decomposition of $G$. So $\pw(G) = \infty$ if and only if every path-decomposition of $G$ has bags of arbitrarily large (but not necessarily infinite) size. Our definition of path-decomposition is a natural extension of the usual definition of path-decomposition for finite graphs and is standard\footnote{Some papers in the literature, such as~\citep{NSS25,NSS-AS-II,NSS-AS-III}, use $\mathbb{Z}$ instead of $\mathbb{N}$ as the indexing set of a path-decomposition.
Using $\mathbb{Z}$ instead of $\mathbb{N}$ as the indexing set does not increase the path-width and decreases it by at most a multiplicative factor of $2$. All our results in this paper remain true even if the path-decomposition is defined by using $\mathbb{Z}$ as the indexing set. Our motivation for using $\mathbb{N}$ as the indexing set is that it gives the chain of inequalities \textcolor{blue!60!black}{(}\ref{observation}\textcolor{blue!60!black}{)}, given in \cref{sec:compare}.}. 

\subsection{Characterisation}

Here we consider which graphs have finite path-width. Compactness arguments for path-width are not straightforward, since the following two graphs have infinite path-width, but every finite subgraph has path-width 1:
\begin{itemize}
\item the disjoint union of infinitely many 1-way infinite paths, and
\item the disjoint union of infinitely many infinite stars.
\end{itemize}
See \cref{DisjointPaths,InfiniteDegreeVertices} for proofs of these results. These examples show that the following conditions are necessary for a graph $G$ to have finite path-width, where the third condition is immediate from the definition of infinite path-width:
\begin{itemize}
    \item $G$ has finitely many vertices of infinite degree, 
    \item $G$ contains finitely many pairwise disjoint 1-way infinite paths\footnote{Up to isomorphism, the \defn{1-way infinite path} is the graph with vertex set $\mathbb{N}$ and edge set $\{\{n, n + 1\} : n \in \mathbb{N}\}$, and the \defn{2-way infinite path} is the graph with vertex set $\mathbb{Z}$ and edge set $\{\{a, a + 1\} : a \in \mathbb{Z}\}$.}, 
    \item the finite subgraphs of $G$ have path-width bounded by an absolute constant. 
\end{itemize}

Our first main result shows that these three necessary conditions are also sufficient.

\begin{thm}
\label{FinitePathwidthQuantitative}
Let $G$ be a graph such that for some integers $d,c\geq 0$ and $k\geq 1$\textnormal{:} 
    \begin{itemize}
        \item at most~$d$ vertices in $G$ have infinite degree,
        \item $G$ does not contain $k$ pairwise disjoint 1-way infinite paths, and
        \item every finite subgraph of~$G$ has path-width at most~$c$.
    \end{itemize}
Then $\pw(G) \leq 2c + (3k-2)(2c + 1) + d$.
\end{thm}

\cref{FinitePathwidthQuantitative}, which is proved in \cref{sec:finite_pw}, implies the following excluded-subgraph characterisation of graphs with finite path-width. Here we also use the above necessary conditions, and the fact that a monotone class $\GG$ of finite graphs has  bounded path-width if and only if for some finite tree $T$ with maximum degree 3, no subdivision of $T$ is in $\GG$ (see \citep{RS-I,BRST91,Diestel95}). 

\begin{cor}
\label{FinitePathwidth}
A graph $G$ has finite path-width if and only if:
    \begin{itemize}
        \item $G$ has finitely many vertices of infinite degree,
        \item $G$ contains finitely many pairwise disjoint 1-way infinite paths, and
        \item there exists a finite tree $T$ of maximum degree 3 such that $G$ contains no subdivision of $T$.
    \end{itemize}
\end{cor}

\cref{FinitePathwidth} shows that path-width is more nuanced than the closely related parameter of tree-width, where the analogous compactness theorem is much simpler to state:   \citet{Thomas88} showed that a graph $G$ has tree-width at most $c$ if and only if every finite subgraph of $G$ has tree-width at most~$c$ (see \citep{KT91,Thomassen89} for simpler proofs).

In light of the above-mentioned sufficient conditions for a graph to have infinite path-width, line-width is a commonly used alternative; see~\citep{DT99,NSS25,NSS-AS-II,NSS-AS-III,CNSS26} and \citep[Chapter~12]{Diestel6}. A \defn{line} is a pair $(L, \preccurlyeq)$, where $L$ is a non-empty set and $\preccurlyeq$ is a total order of $L$. A \defn{line-decomposition} (also called a `\textit{linear decomposition}' in \citep{DT99,Diestel6}) of a graph $G$ is a pair $((L, \preccurlyeq), (B_{t} : t \in L))$ such that $(L, \preccurlyeq)$ is a line and:

\begin{itemize}
    \item $\bigcup_{t \in L}B_{t} = V(G)$,
    \item for each $vw \in E(G)$, there exists $t \in L$ such that $v, w \in B_{t}$, and
    \item $B_{t_{1}} \cap B_{t_{3}} \subseteq B_{t_{2}}$ for all $t_{1}, t_{2}, t_{3} \in L$ with $t_{1} \preccurlyeq t_{2} \preccurlyeq t_{3}$.
\end{itemize}

The \defn{width} of a line-decomposition $((L, \preccurlyeq), (B_{t} : t \in L))$ is $\sup_{t \in L}(|B_{t}| - 1)$. The \defn{line-width} of a graph $G$, denoted by \defn{$\lw(G)$}, is the minimum width of a line-decomposition of $G$. In the case where $L = \mathbb{N}$ and $\preccurlyeq$ is the standard total ordering $\leqslant$ of $\mathbb{N}$, a line-decomposition $((\mathbb{N}, \leqslant), (B_{t} : t \in \mathbb{N}))$ is equivalent to a path-decomposition $(B_t : t \in \mathbb{N})$. So $\lw(G) \leqslant \pw(G)$ for every graph $G$, with equality if $G$ is finite.

\citet*{CNSS26} proved the following compactness result for  line-width, which we use in the proof of \cref{FinitePathwidthQuantitative}.

\begin{thm}[{\protect\citep[Proposition~4.1]{CNSS26}}] 
\label{FiniteLinewidth}
For any graph $G$ and integer $c$, $G$ has line-width at most $c$ if and only if every finite subgraph of $G$ has path-width at most $c$.
\end{thm}

\subsection{Universality}

Now consider graph universality. A graph $U$ is \defn{universal}\footnote{Unlike some other papers, we do not require $U \in \mathcal{G}$ for $U$ to be universal for $\mathcal{G}$.} for a graph class $\mathcal{G}$ if $U$ contains every graph in $\mathcal{G}$.

A strong property of a universal graph $U$ for a graph class $\mathcal{G}$ is that $U$ is in $\mathcal{G}$.  For instance, it is straightforward to construct a tree that is universal for the class of all trees 
(see \citep[Theorem~4.1]{HMSTW} for example).  More generally, if $\mathcal{T}_k$ is the class of graphs with tree-width at most $k$, then there exists a universal graph for $\mathcal{T}_k$ that is in $\mathcal{T}_k$.  This follows from the theory of simplicial decompositions~\citep{Thom83,Diestel90,Halin84,Halin82,Halin78}. An alternative and explicit definition is also given in~\citep{HMSTW}. Another example is the \textit{Rado graph} \citep{Ackermann37,ER63a,Rado64}\footnote{Up to isomorphism, the \defn{Rado graph} is the graph with vertex set $\mathbb{N}_{0}$ such that two vertices $x, y \in \mathbb{N}_0$ with $x < y$ are adjacent if and only if the $x$-th bit of the binary representation of $y$ is non-zero. The Rado graph is often called the \defn{random graph}~\citep{Cameron01,Cameron97,Cameron84}, since it has a probabilistic definition given by \citet{ER63a}.}, which contains every graph as an \textit{induced} subgraph.

However, for some graph classes $\mathcal{G}$, there is no universal graph that belongs to $\mathcal{G}$. For example, an old question of Ulam  asks whether there exists a planar graph that is universal for the class of all planar graphs. \Citet{Pach81a} answered Ulam's question in the negative by showing that every universal graph for the class of planar graphs has a $K_5$-minor or a $K_{3, 3}$-minor. Recently,~\citet*{HMSTW} extended Pach's result by showing that every universal graph for the class of planar graphs has an infinite complete graph minor.  In this sense, every universal graph for the class of planar graphs is `far' from being planar. 

If a graph class $\mathcal{G}$ has no universal graph in $\mathcal{G}$, then it is interesting to establish how `close' a universal graph $U$ for $\mathcal{G}$ can be to $\mathcal{G}$. For example, \citet{HMSTW} constructed a universal graph $U$ for the class of planar graphs, where $U$ has linear expansion, and every finite $n$-vertex subgraph of $U$ has a balanced separator of size $\mathcal{O}(\sqrt{n})$. In this sense, $U$ is `close' to being planar.

See \citep{Henson71,Komjath99,HP84,CS07a,CS16,FK97a,FK97b,CS01,KMP88,KP84,CT07,CSS99,DHV85,Diestel85,BHM13,MMS09,Lehner23,Lehner24,Krill23,Krill25,GH23,Georgakopoulos25,BEGHRW} for more results on universality. We refer the interested reader to~\citep[Section 1.4]{HMSTW} for a short survey.

Now consider universality for graphs with given path-width. For each $k \geqslant 0$, let \defn{$\mathcal{P}_{k}$} be the class of graphs with path-width at most $k$. Our first universality result shows that no graph of finite path-width is universal for $\mathcal{P}_1$, even restricted to the class of locally finite graphs with path-width~$1$. Here a graph is \defn{locally finite} if each vertex has finite degree. 

\begin{thm} \label{nouniversalpathwidth} 
Every universal graph for the class of locally finite graphs with path-width $1$ has path-width $\infty$.
\end{thm}

\cref{nouniversalpathwidth} is proved in \cref{InfinitePathwidth}. It implies that for $k \geqslant 1$, if $U$ is universal for $\mathcal{P}_k$, then $U$ is not in $\mathcal{P}_{k}$. In fact, even in the $k=1$ case, the path-width of $U$ must be $\infty$. In this sense, $U$ is `far' from $\mathcal{P}_{k}$. This is in stark contrast to the above-mentioned universality result for graphs with given tree-width. In this sense, universality for graphs with given path-width is subtler than universality for graphs with given tree-width.

Now consider universality for graphs with given line-width. For each $k \geqslant 0$, let \defn{$\mathcal{L}_{k}$} be the class of graphs with line-width at most~$k$. So $\mathcal{P}_k \subseteq \mathcal{L}_k$. The following result constructs a universal graph for $\mathcal{L}_{k}$ that is `close' to~$\mathcal{L}_{k}$.

\begin{restatable}{thm}{universallw}
\label{thm:universallw}
For every integer $k \geq 0$, there is a universal graph for $\mathcal{L}_k$ with line-width at most $4k^2+6k$.
\end{restatable}

Complementing \cref{thm:universallw}, the next result implies that for $k \geqslant 2$, no universal graph for $\mathcal{L}_k$ is in $\mathcal{L}_k$.

\begin{restatable}{thm}{thmlowerbound} 
\label{thm:lowerbound} 
    For every integer $k \geqslant 2$, every universal graph for $\mathcal{P}_k$ has line-width at least $k + 1$.
\end{restatable}

Note that $k \geqslant 2$ is necessary in~\cref{thm:lowerbound}, since there is a graph with line-width $1$ that is universal for $\mathcal{L}_1$ (see~\cref{prop:lw1universal}).

The proofs of~\cref{thm:universallw,thm:lowerbound}, which are presented in \cref{sec:proof2,sec:proof3} respectively, are inspired by analogous results for finite graphs by \citet*{BEGHRW}.

\section{Path-width, Line-width and Well-order-width} \label{sec:compare}

This section introduces well-order-width, which is a tool used throughout the paper. In a line-decomposition, bags are indexed by an arbitrary totally ordered set, which can make them difficult to work with. To tackle this issue, \citet*{NSS25} introduced well-order-decompositions (which they called `\textit{wo-decompositions}'). A total order $\preccurlyeq$ of a non-empty set $L$ is a \defn{well-order} if there is no infinite sequence $t_{1}, t_{2}, \dots$ of distinct elements of $L$ such that $t_{i + 1} \preccurlyeq t_{i}$ for each $i \geqslant 1$. 
Equivalently, every non-empty subset $A$ of $L$ has an element $a \in A$, called the \defn{least element} of $A$, such that $a \preccurlyeq b$ for every $b \in A$.  A \defn{well-order-decomposition} is a line-decomposition $((L, \preccurlyeq), (B_{t} : t \in L))$ such that $\preccurlyeq$ is a well-order of $L$. The \defn{well-order-width} (called `\textit{wo-width}' in \citep{NSS25}) of a graph $G$, denoted by \defn{$\wow(G)$}, is the minimum width of a well-order-decomposition of $G$. For each $k \geqslant 0$, let \defn{$\mathcal{W}_{k}$} be the class of graphs with well-order-width at most $k$.

Observe that for every graph $G$,
\begin{equation} \label{observation} \tag{1}
\lw(G) \leqslant \wow(G) \leqslant \pw(G).
\end{equation}
Moreover, \citet[1.3]{NSS25} proved the following result.

\begin{thm}[\citep{NSS25}] \label{lwandwow} 
    For every graph $G$,
    \[\wow(G) \leqslant 2\lw(G).\]
\end{thm}
So line-width and well-order-width are within a factor 2 of each other. Our proof of \cref{thm:universallw}, given in \cref{sec:proof2}, constructs a universal graph for $\mathcal{W}_{k}$ that is `close' to $\mathcal{W}_{k}$ and applies \cref{lwandwow}.

We finish this section with the following elementary lemma used later. 

\begin{lem} \label{lem:unionwow} 
The well-order-width of any graph $G$ is the supremum of the well-order-widths of the connected components of $G$.  
\end{lem}

\begin{proof}
Since well-order-width is monotone under taking subgraphs, the well-order-width of $G$ is at least the supremum of the well-order-widths of the connected components of $G$. Now we prove the converse. 
If the supremum is infinite, then the result is immediate.
Otherwise, let $H_{1}, H_{2}, \dots$ be the connected components of $G$ (this collection of components might be finite or infinite) and note that their well-order-widths have a finite maximum. 
It suffices to construct a well-order-decomposition of $G$ with width $\max(\wow(H_{1}), \wow(H_{2}), \dots)$. For each component $H_{i}$, there is a well-order-decomposition $((L_{i}, \preccurlyeq_{i}), (B_{t}^{i} : t \in L_{i}))$ of $H_{i}$ with width $\wow(H_{i})$. We may and will assume that $L_{1}, L_{2}, \dots$ are pairwise disjoint. Let $L := L_{1} \cup L_{2} \cup \dots$ and $\preccurlyeq$ be the concatenation of $(\preccurlyeq_{1}$, $\preccurlyeq_{2}, \dots)$. That is, for each $t_1, t_2 \in L$, we have $t_1 \preccurlyeq t_2$ if and only if: (i) $t_1 \in L_{i}$ and $t_2 \in L_{j}$ for some $i, j \in \mathbb{N}$ with $i < j$, or (ii) $t_1 \in L_{i}$, $t_2 \in L_{i}$ and $t_1 \preccurlyeq_{i} t_2$ for some $i \in \mathbb{N}$. Since $\preccurlyeq_{i}$ is a well-order of $L_{i}$ for any component $H_{i}$, we have that $\preccurlyeq$ is a well-order of $L$. For each component $H_{i}$ and each $t \in L_{i}$, let $B_{t} := B_{t}^{i}$. So for each $t \in L$, we have $|B_{t}| \leqslant \max(\wow(H_{1}), \wow(H_{2}), \dots) + 1$. Hence $((L, \preccurlyeq), (B_{t} : t \in L))$ is a well-order-decomposition of $G$ with width $\max(\wow(H_{1}), \wow(H_{2}), \dots)$. This completes the proof.
\end{proof}

An analogous proof shows the following. 

\begin{lem} \label{lem:unionlw} 
The line-width of any graph $G$ is the supremum of the line-widths of the connected components of $G$.  
\end{lem}

 \section{Graphs with Infinite Path-width}
 \label{InfinitePathwidth}


This section proves the necessary conditions for a graph to have finite path-width introduced in \cref{sec:intro}. 

\begin{lem} \label{lem:inifinite_conn}
For any graph $G$ and integer $w\geq 1$, if there exist $w$ pairwise disjoint subsets $S_1,S_2,\dots,S_w$ of $V(G)$ such that $G[S_i]$ is connected and $N_G[S_i]$ is infinite for every $i\in\{1,\dots,w\}$, then $\pw(G) \geq w$.
\end{lem}

\begin{proof}
Suppose to the contrary that there exists a path-decomposition $(B_{t} : t \in \mathbb{N})$ of $G$ of width at most $w-1$. 
For $i\in\{1,\dots,w\}$, let $I_i=\{t \in \mathbb{N}: B_t \cap S_i \neq \emptyset\}$.
Then each $I_i$ is an interval in $\mathbb{N}$ since $G[S_i]$ is connected.
By assumption, for each $i\in\{1,\dots,w\}$, there are infinitely many edges with at least one endpoint in $S_i$.   Since each bag is finite, $I_i$ is infinite for every $i\in\{1,\dots,w\}$.  Let $m:=\max_{i \in \{1, \dots, w\}} \min I_i$.  
Thus, $S_i \cap B_m \neq \emptyset$ for all $i \in \{1, \dots, w\}$.  Since $S_1, \dots, S_w$ are pairwise disjoint, $|B_t|\geq w$ for all $t \geq m$.  Suppose $S_1$ is infinite. Since $\bigcup_{i=1}^m B_i$ is finite, there exists $uv \in E( G[S_1] )$ such that  $\{u,v\} \not \subseteq \bigcup_{i=1}^m B_i$.  Thus, $\{u,v\} \subseteq B_{t'}$ for some $t' > m$, which implies that $|B_{t'}| \geq w+1$. Thus, we may assume that each $S_i$ is finite.  But now, there exists $uv \in E(G)$ such that $u \in S_1$, $v \notin \bigcup_{i=2}^w S_i$, and  $\{u,v\} \not \subseteq \bigcup_{i=1}^m B_i$.  Again, this implies that $|B_{t'}| \geq w+1$ for some $t'>m$.  
\end{proof}

\cref{lem:inifinite_conn} implies the following result 
by taking $S_i=V(P_i)$ for each $i\in\{1,\dots,k\}$.

\begin{cor} 
\label{DisjointPaths}
If a graph $G$ contains $k$ pairwise disjoint 1-way infinite paths $P_1,\dots,P_k$, then $\pw(G)\geq k$. 
If a graph $G$ contains infinitely many pairwise disjoint 1-way infinite paths, then $\pw(G)=\infty$.
\end{cor}

Note that $\pw(G)\geq k$ is tight in \cref{DisjointPaths}, since $P_1\sqcup\dots\sqcup P_k$ has a path-decomposition with width $k$, where each bag consists of the ends of some edge in some $P_i$ and one vertex from every other $P_j$. 

Similarly, \cref{lem:inifinite_conn} implies the following result 
by taking $S_i=\{v_i\}$ for each $i\in\{1,\dots,k\}$.

\begin{cor} 
\label{InfiniteDegreeVertices}
If a graph $G$ has $k$ vertices $v_1,\dots,v_k$ of infinite degree, then $\pw(G)\geq k$. If a graph $G$ has infinitely many vertices of infinite degree, then $\pw(G)=\infty$.
\end{cor}

Note that $\pw(G)\geq k$ is tight in \cref{InfiniteDegreeVertices}, since the complete bipartite graph $K_{k,\aleph_0}$ has path-width $k$.

Recall that \cref{nouniversalpathwidth} says that every universal graph for the class of locally finite graphs with path-width 1 has path-width $\infty$. This immediately follows from \cref{InfiniteDegreeVertices} and the next result.

\begin{lem} \label{lem:infdegree}
Every graph $U$ that contains every locally finite graph with path-width $1$ has infinitely many vertices of infinite degree. 
\end{lem}
\begin{proof}
Let $I$ be the set of vertices of infinite degree in $U$. Suppose for the sake of contradiction that $I$ is finite; let $k := |I|$. Say $V(U) = \set{v_1,v_2,\dots}$, where $I = \set{v_1,\dots,v_k}$, so $\deg_U(v_i)$ is finite for each $i \geq k + 1$. For each $i \geqslant k + 1$ and $d \geqslant 0$, let $n_{i,d}$ be the maximum degree in $U$ of a vertex at distance at most $d$ from $v_i$ in $U - I$. Note that $n_{i,d}$ is finite since there are finitely many vertices at distance at most $d$ from $v_i$ in $U - I$, all of which have finite degree.

For each integer $i \geqslant 1$, let $f(i) := \max\set{n_{j,j} \colon k + 1\leq j\leq k + i}$. Note that $(f(1), f(2), \dots)$ is a non-decreasing sequence. Consider a caterpillar $G$ with a 1-way infinite spine $(t_1, t_2, \dots)$, where each vertex $t_i$ is adjacent to $f(i)$ leaves. So $G$ is locally finite. Every locally finite caterpillar with a 1-way infinite spine has path-width $1$, and thus $G$ has path-width~1 (by a similar argument given in the proof of~\cref{prop:caterpillar}). By assumption, there is an isomorphism $\phi$ from $G$ to a subgraph of $U$. Since $I$ is finite, there is a positive integer $\ell$ such that $\phi(t_i) \notin I$ for all $i \geq \ell$. 

Let $j^\ast$ be the integer such that $v_{j^\ast} = \phi(t_{\ell})$.
So $j^\ast \geq k + 1$. 
The path $\phi(t_{\ell}) \phi(t_{\ell + 1}) \dots \phi(t_{\ell + j^\ast})$ lies in $U - I$. So $\phi(t_{\ell + j^\ast})$ is at distance at most $j^\ast$ from $\phi(t_{\ell}) = v_{j^\ast}$ in $U - I$. Therefore, $\deg_U(\phi(t_{\ell + j^\ast})) \leq n_{j^\ast, j^\ast}$. Since $t_{\ell + j^\ast}$ is adjacent to $f(\ell + j^\ast)$ leaves, we have $\deg_{G}(t_{\ell + j^\ast}) = f(\ell + j^\ast) + 2$. Thus, 
\[
f(\ell + j^\ast) + 2 = \deg_G(t_{\ell + j^\ast}) \leq \deg_U(\phi(t_{\ell + j^\ast})) \leq n_{j^\ast, j^\ast} \leq f(j^\ast).
\]
This contradicts the fact that $(f(1), f(2), \dots)$ is a non-decreasing sequence. Therefore $I$ is infinite, as required.
\end{proof}

 \section{Graphs with Line-width~1} \label{section:proof1}

This section presents a series of examples of graphs with line-width 1 and other interesting properties, culminating in a construction of a graph with line-width 1 that is universal for the class of graphs with line-width at most 1. We start with the following lemma, implicitly proved by \citet[3.1]{NSS25}. We include the proof for completeness.

\begin{lem} [\citep{NSS25}] \label{lem:wowpath} 
    The $2$-way infinite path has line-width $1$ and well-order-width $2$.
\end{lem}

\begin{proof}
Let $P$ denote the $2$-way infinite path, so $V(P) = \mathbb{Z}$ and $E(P) = \{\{a, a + 1\} : a \in \mathbb{Z}\}$. For each $a \in \mathbb{Z}$, let $R_{a} := \{a, a + 1\}$. Observe that $((\mathbb{Z}, \leqslant), (R_{a} : a \in \mathbb{Z}))$ is a line-decomposition with width $1$ of $P$. Since $P$ has edges, $\lw(P) = 1$.

By \cref{lwandwow},  $\wow(P) \leqslant 2$. We now show that $\wow(P) = 2$. Suppose for the sake of contradiction that $((L, \preccurlyeq), (B_{t} : t \in L))$ is a well-order-decomposition of $P$ with width~$1$. Let $T := \{t \in L : B_{t} = \{a, a + 1\} \text{ for some } a \in \mathbb{Z}\}$. Since $|B_{t}| \leqslant 2$ for each $t \in L$ and $\{a, a + 1\} \in E(P)$ for each $a \in \mathbb{Z}$, we have that $T$ is non-empty (and, in fact, $T$ is infinite). Since $\preccurlyeq$ is a well-order of $L$, the set $T$ has a least element $t'$. So $B_{t'} = \{a, a + 1\}$ for some $a \in \mathbb{Z}$. Let $t_{1}, t_{2} \in L$ such that $B_{t_1} = \{a - 2, a - 1\}$ and $B_{t_2} = \{a + 2, a + 3\}$. Note that $t_1, t_2 \in T$. By the choice of $t'$, we have $t' \preccurlyeq t_{1}$ and $t' \preccurlyeq t_{2}$. Suppose that $t' \preccurlyeq t_{1} \preccurlyeq t_{2}$. Since $\{a + 1, a + 2\} \in E(P)$, for each $t^\ast \in L$ such that $t' \preccurlyeq t^\ast \preccurlyeq t_{2}$, we have $a + 1 \in B_{t^\ast}$ or $a + 2 \in B_{t^\ast}$. Setting $t^\ast = t_{1}$, we reach a contradiction. So it is not the case that $t' \preccurlyeq t_{1} \preccurlyeq t_{2}$. Similarly, it is not the case that $t' \preccurlyeq t_{2} \preccurlyeq t_{1}$ because $\{a - 1, a\} \in E(P)$. Therefore $\wow(P) = 2$.
\end{proof}

Up to isomorphism, the \defn{infinite star} is the graph with vertex set $\mathbb{N}$ and edge set $\{\{1, n\} : n \in \mathbb{N} \setminus \{1\}\}$. The next proposition separates line-width (and well-order-width) from path-width. 

\begin{restatable}{prop}{propinfinitestars} \label{prop:stars} 
If a graph $G$ is the disjoint union of infinitely many infinite stars, then $G$ has line-width $1$, well-order-width $1$, and path-width $\infty$.
\end{restatable}

\begin{proof} We first show that the infinite star $S$ has line-width $1$, well-order-width $1$, and path-width $1$. We assume that $V(S) = \mathbb{N}$ and $E(S) = \{\{1, n\} : n \in \mathbb{N} \setminus \{1\}\}$. For each $t \in \{1, 2, \dots\}$, let $B_{t} := \{1, t + 1\}$. Note that $((\mathbb{N}, \leqslant), (B_{t} : t \in \mathbb{N}))$ is a line-decomposition and well-order-decomposition of $S$ with width $1$, and $(B_{t} : t \in \mathbb{N})$ is a path-decomposition of $S$ with width $1$. Since $S$ has edges, $\lw(S) = \wow(S) = \pw(S) = 1$.

By \cref{lem:unionwow}, $\wow(G) = 1$. By \textcolor{blue!60!black}{(}\ref{observation}\textcolor{blue!60!black}{)}, 
$\lw(G) = 1$. By \cref{InfiniteDegreeVertices}, $\pw(G) = \infty$.
\end{proof}

\Cref{prop:stars} implies that path-width cannot be bounded from above by a function of line-width (or well-order-width). The next proposition gives a modified version of this example. A \defn{caterpillar} is a tree $G$ obtained from a path $P$ (called a \defn{spine} of $G$) by adding leaves adjacent to $P$.  It is straightforward to show that a finite connected graph has path-width $1$ if and only if it is a finite caterpillar. The next proposition considers a generalisation of this to the infinite setting.

\begin{restatable}{prop}{propcaterpillar} \label{prop:caterpillar} 
If $G$ is a caterpillar with a 2-way infinite spine where every vertex of the spine has infinite degree, then $G$ has line-width 1, well-order-width 2, and path-width~$\infty$.
\end{restatable}

\begin{proof}
We may and will assume that $V(G)=\mathbb{Z} \cup (\mathbb{Z} \times \mathbb{Z})$, where $xy \in E(G)$ if and only if: (i) $x, y \in \mathbb{Z}$ and $|x - y| = 1$, or (ii) $x \in \mathbb{Z}$ and $y = (x, i)$ for some $i \in \mathbb{Z}$, or (iii) $y \in \mathbb{Z}$ and $x = (y, i)$ for some $i \in \mathbb{Z}$.

Let $\preccurlyeq$ be the total order of $\mathbb{Z} \times (\mathbb{Z} \cup \{\infty\})$, where $(a, b) \preccurlyeq (a', b')$ if and only if (i) $a < a'$, or (ii) $a = a'$ and $b \leqslant b'$.  
For each $(a,b) \in \mathbb{Z} \times \mathbb{Z}$, let $B_{(a,b)}:=\{a\} \cup \{(a,b)\}$. For each $a \in \mathbb{Z}$, let $B_{(a,\infty)} := \{a, a + 1\}$.  Observe that $((\mathbb{Z} \times (\mathbb{Z} \cup \{\infty\}), \preccurlyeq), (B_t : t \in \mathbb{Z} \times (\mathbb{Z} \cup \{\infty\})))$ is a line-decomposition of $G$ of width 1. Since $G$ has edges, $\lw(G) = 1$.
    
    By \cref{lwandwow}, $\wow(G) \leq 2$. Note that the $2$-way infinite path is a spine of $G$, and hence is a subgraph of $G$. Thus, by \cref{lem:wowpath}, $\wow(G) = 2$.

    Finally, by \cref{InfiniteDegreeVertices}, $\pw(G) = \infty$.
\end{proof}

\begin{prop} \label{prop:lw1universal}
There exists a graph $U$ with line-width $1$ that is universal for the class of graphs with line-width at most $1$.
\end{prop}

\begin{proof}
Let $G$ be a caterpillar with a 2-way infinite spine such that each vertex of the spine has infinite degree. Let $U$ be the disjoint union of countably many copies of $G$. By~\cref{lem:unionlw} and~\cref{prop:caterpillar}, $U$ has line-width 1. 

Let $H$ be a graph with line-width at most 1.   Thus, every finite subgraph of $H$ has path-width at most 1.  In particular, $H$ contains neither a cycle nor the 1-subdivision of $K_{1,3}$.   

We now show that $H$ is isomorphic to a subgraph of $U$. Let $T$ be a connected component of $H$.
It suffices to show that $T$ is a caterpillar.  Since $T$ contains no cycle, $T$ is a tree.  Let $T'$ be obtained from $T$ by deleting the leaves of $T$.  We claim that $T'$ has maximum degree at most 2. Towards a contradiction, suppose $v \in V(T')$ has distinct neighbours $u_1,u_2,u_3 \in V(T')$.  Since none of $u_1, u_2,u_3$ are leaves of $T$, $T$ contains a $1$-subdivision of $K_{1,3}$, which is a contradiction.  Thus either $T'$ is empty (and so $T$ is $K_1$ or $K_2$) or $T'$ is connected with maximum degree at most 2 and so $T'$ is a path. In either case, $T$ is a caterpillar.  
\end{proof}

\section{Finite Path-width: Proof of \texorpdfstring{\cref{FinitePathwidthQuantitative}}{Theorem 1}} 
\label{sec:finite_pw}

This section proves \cref{FinitePathwidthQuantitative}, which characterises graphs with finite path-width. Let $\preccurlyeq$ be a well-order of a set $L$. For a set $S \subseteq L$, the \defn{supremum} of $S$ is the least element of $\{t \in L : s \preccurlyeq t \text{ for every } s \in S\}$. Since $\preccurlyeq$ is a well-order of $L$, the supremum of $S$ exists if and only if $\{t \in L : s \preccurlyeq t \text{ for every } s \in S\} \neq \emptyset$.

\begin{lem} \label{lem:NoFiniteComp}
Let $k$ and $w$ be positive integers.
Let $G$ be a locally finite graph that does not contain $k$ pairwise disjoint $1$-way infinite paths, and every connected component of $G$ contains a $1$-way infinite path.
If $\wow(G) \leqslant w$ then $\pw(G) \leqslant w + 3(k - 1)(w + 1)$.
\end{lem}

\begin{proof}
Let $((L,\preccurlyeq),(B_t: t \in L))$ be a well-order-decomposition of $G$ of width at most~$w$.
If necessary, we may add an element into $L$ that is strictly greater than every other element in $L$
and define the bag of this added element to be the empty set.
This still keeps $((L, \preccurlyeq),(B_t: t \in L))$ a well-order-decomposition of $G$ of width at most $w$, but now the supremum of $S$ exists for every $S \subseteq L$.

For every $v \in V(G)$, let $h_v$ be the least element of $\{x \in L: v \in B_x\}$, which exists since $\preccurlyeq$ is a well-order of $L$. Note that $v \in B_{h_{v}}$.

Let $\C$ be a maximal collection of pairwise disjoint 1-way infinite paths in $G$.
By assumption of this lemma, $|\C| \leq k-1$.
Since every component of $G$ contains a $1$-way infinite path, $\C \neq \emptyset$.

For each $P \in \C$, we denote $P$ by $v_{P,1}v_{P,2} \dotsb$, and we define $P(n) := P - \{v_{P,i}: 1 \leqslant i < n\}$ for every positive integer $n$.
For any elements $x, y \in L$ with $x \preccurlyeq y$, denote $[x,y] := \{t \in L: x \preccurlyeq t \preccurlyeq y\}$ and $[x,y) := \{t \in L: x \preccurlyeq t \prec y\}$.

\setcounter{claim}{0}

\begin{claim} \label{claim:1/8} For every $P \in \C$, there exists $m_P \in L$ such that \textnormal{(i)} $m_{P}$ is not the least element of $L$, and \textnormal{(ii)} for every $x \in L$ with $x \prec m_P$, there exists a positive integer $N$ such that $h_v \in [x,m_P]$ for every $v \in V(P(N))$. In particular, $V(P(N)) \subseteq \bigcup_{y \in [x,m_P]}B_y$.
\end{claim}

\begin{proof}

Fix $P \in \C$.
For every positive integer $n$, let $s_n$ be the supremum of $\{h_v: v \in V(P(n))\}$. 
Let $m_P$ be the least element of $\{s_n : n \in \mathbb{N}\}$, which exists since $\preccurlyeq$ is a well-order of $L$.
Observe that $s_1 \succcurlyeq s_2 \succcurlyeq \dotsb$.
So there exists $t \in \mathbb{N}$ such that $m_{P} = s_i$ for every $i \geq t$. 

Suppose for the sake of contradiction that $m_{P}$ is the least element of $L$. 
By the definition of $s_t$, we have $h_{v} = m_{P}$ for every $v \in V(P(t))$. Since $v \in B_{h_{v}}$, we have $V(P(t)) \subseteq B_{m_{P}}$. So $B_{m_{P}}$ has infinite size. This contradicts the assumption that the width of $((L, \preccurlyeq), (B_t: t \in L))$ is at most $w$. Therefore $m_{P}$ is not the least element of $L$, and this shows (i).

We now show (ii). Let $x \in L$ such that $x \prec m_P$.
Since $m_P$ is the least element of $\{s_n : n \in \mathbb{N}\}$
, we have that $x \prec s_n$ for every positive integer $n$.
Thus there exist infinitely many positive integers $n_1<n_2<\dots$ such that $x \prec h_{v_{P,n_i}}$ for every $i \geq 1$.

Suppose for the sake of contradiction that there does not exist a positive integer $N$ such that $h_v \in [x,m_P]$ for every $v \in V(P(N))$. 
Then there exist infinitely many positive integers $a_{1} < a_{2} < \dots$ such that for every $i \geq 1$, either $h_{v_{P,a_i}} \prec x$ or $m_P \prec h_{v_{P,a_i}}$. Hence exactly one of the following two cases holds.

If there exists a positive integer $j$ such that $m_P \prec h_{v_{P,a_i}}$ for every $i \geq j$, then $m_P \prec s_n$ for every positive integer $n \geq a_j$, contradicting that $m_P=s_i$ for every $i \geq t$.

So there are infinitely many integers $i \geq 1$ such that $h_{v_{P,a_i}} \prec x$.
Hence there exist infinitely many positive integers $i_1, j_1, i_2, j_2, \dots$ such that $a_{i_1} < n_{j_1} < a_{i_2} < n_{j_2} < \dots$ and $h_{v_{P, a_{i_k}}} \prec x$ for every integer $k \geqslant 1$.

For every integer $k \geqslant 1$, let $P_{k}$ be the subpath of $P$ with endpoints $v_{P, a_{i_k}}$ and $v_{P, n_{j_k}}$. Since $a_{i_1} < n_{j_1} < a_{i_2} < n_{j_2} < \dotsb$, we have that $P_{1}, P_{2}, \dots$ are pairwise disjoint.
For every integer $k \geqslant 1$, we have that $B_x \cap V(P_k) \neq \emptyset$ since $h_{v_{P,a_{i_k}}} \prec x \prec h_{v_{P,n_{j_k}}}$.
Hence $|B_x|$ is infinite, a contradiction.
\end{proof}

Denote the members of $\C$ by $P_1,P_2,\dots,P_{|\C|}$ such that $m_{P_1} \preccurlyeq m_{P_2} \preccurlyeq \dots \preccurlyeq m_{P_{|\C|}}$.
For every $i \in \{1, \dots, |\C|\}$, let $m_i := m_{P_i}$. Thus $m_{1} \preccurlyeq \dots \preccurlyeq m_{|\C|}$.

For every $v \in V(G)$, let $h'_v$ be the supremum of $\{t \in L: v \in B_t\}$. For every $t \in L$, let $B_t' := B_t \cup \{v \in V(G): h'_v=t\}$. Observe that $|B_t'| \leq 2(w+1)$.

\begin{claim} \label{claim:2/8} For every $i \in \{1, \dots, |\C|\}$, there exists $\ell_{i} \in L$ such that 
    \begin{itemize}
        \item[\textnormal{(i)}] $\ell_i \prec m_i$, 
        \item[\textnormal{(ii)}] $\{m_P: P \in \C\} \cap [\ell_i,m_i] \subseteq \{m_i\}$, and 
        \item[\textnormal{(iii)}] for every $P \in \C$ with $m_P = m_i$ and for every $y \in L$ with $\ell_i \preccurlyeq y \prec m_i$, $B_y \setminus B_{m_i}'$ intersects the infinite subpath of $P-B_{m_i}'$.
    \end{itemize}
\end{claim}

\begin{proof}

Let $\alpha$ be the least element of $L$, which exists since $\preccurlyeq$ is a well-order of $L$. By \cref{claim:1/8}(i), $\alpha \prec m_{1} \preccurlyeq \dots \preccurlyeq m_{|\C|}$.

Fix $i \in \{1, \dots, |\C|\}$. 
Let $A := (\set{\alpha} \cup \set{m_P : P \in \C}) \cap [\alpha, m_{i})$. 
Since $A$ is finite, it has a greatest element $\beta$; that is, $\beta \in A$ and $x \preccurlyeq \beta$ for every $x \in A$. 
By construction of $A$, we have $\beta \prec m_{i}$ and $\{m_P: P \in \C\} \cap [\beta, m_{i}] \subseteq \{\beta, m_{i}\}$.

Since $\beta \prec m_i$, \cref{claim:1/8}(ii) implies that there exists a positive integer $N$ such that $V(P_i(N)) \subseteq \bigcup_{r \in [\beta, m_i]}B_r$.
Since the width of $((L,\preccurlyeq),(B_t: t \in L))$ is at most $w$, we have $|B_{\beta} \cup B_{m_i}| \leqslant 2w+2$. Thus $[\beta, m_i] \setminus \{\beta, m_i\} \neq \emptyset$.

Let $\mathcal{P} := \{P \in \C : m_{P} = m_{i}\}$, and let $P \in \mathcal{P}$. 
Since $|B_{m_{i}}'|$ is finite, $P - B_{m_i}'$ contains a unique infinite $1$-way subpath $P'$ of $P$. By \cref{claim:1/8}(ii), there exists $x_{P} \in [\beta, m_i] \setminus \{\beta, m_i\}$ such that $B_{x_{P}} \cap V(P') \neq \emptyset$. 
Since $B_{x_P}$ is finite, \cref{claim:1/8}(ii) implies that there exists an infinite 1-way subpath $P''$ of $P'$  such that $V(P'') \subseteq (\bigcup_{r \in [x_P,m_i]}B_r) \setminus B_{m_i}'$ and $B_{x_{P}} \cap V(P'') \neq \emptyset$.
By \cref{claim:1/8}(ii), since $P''$ is connected in $G$, the definition of well-order-decomposition implies that $B_{y} \cap V(P'') \neq \emptyset$ for every $y \in L$ with $x_{P} \preccurlyeq y \prec m_P = m_i$. 

So for every $P \in \mathcal{P}$ and for every $y \in L$ with $x_P \preccurlyeq y \prec m_P = m_i$, we have $B_{y} \cap V(P') \setminus B_{m_{i}}' \neq \emptyset$. 
Let $\ell_{i}$ be the greatest element of the non-empty finite set $\{x_{P} : P \in \mathcal{P}\}$. 
So for every $y \in L$ with $\ell_i \preccurlyeq y \prec m_i$, we have $B_{y} \cap V(P') \setminus B_{m_{i}}' \neq \emptyset$. 
This shows (iii). Recall that $\{m_P: P \in \C\} \cap [\beta, m_{i}] \subseteq \{\beta, m_{i}\}$ and for every $P \in \mathcal{P}$, we have $x_{P} \in [\beta, m_i] \setminus \{\beta, m_i\}$. This shows (i) and (ii) and completes the proof of \cref{claim:2/8}.
\end{proof}

Notice that for every $i$, replacing $\ell_i$ by any $\beta$ with $\ell_i \preccurlyeq \beta \prec m_i$ keeps Claim 2 valid. 
So by repeatedly replacing $\ell_i$ by $\max\{\ell_j: 1 \leq j \leq |\C|, m_i=m_j\}$ for some $1 \leq i \leq |\C|$, we may assume that $\ell_a=\ell_b$ whenever $m_a=m_b$ for any $1 \leq a \leq b \leq |\C|$.
In particular, for any $1 \leq i \leq j \leq |\C|$, either $[\ell_i,m_i]=[\ell_j,m_j]$ or $[\ell_i,m_i] \cap [\ell_j,m_j]=\emptyset$.

\begin{claim} \label{claim:3/8} $|V(P_i) \setminus \bigcup_{y \in [\ell_i,m_i]}B_y|$ is finite for every $i \in \{1, \dots, |\C|\}$.
\end{claim}

\begin{proof}
Fix $i \in \{1, \dots, |\C|\}$. By \cref{claim:2/8}(i), $\ell_{i} \prec m_{i}$. By \cref{claim:1/8}(ii) applied to $P_i$ and $\ell_i$, there exists a positive integer $N$ such that $V(P_i(N)) \subseteq \bigcup_{y \in [\ell_i, m_i]}B_y$. Since $|V(P_{i}) \setminus V(P_i(N))|$ is finite, $|V(P_i) \setminus \bigcup_{y \in [\ell_i,m_i]}B_y|$ is finite.
\end{proof}

\begin{claim} \label{claim:3'/8} $|V(P_i) \cap \bigcup_{y \in [\ell_j,m_j]}B_y|$ is finite for any $i,j \in \{1, \dots, |\C|\}$ with $m_i \neq m_j$.
\end{claim}

\begin{proof}
Suppose to the contrary that $|V(P_i) \cap \bigcup_{y \in [\ell_j,m_j]}B_y|$ is infinite.
By \cref{claim:3/8}, $|V(P_i) \setminus \bigcup_{y \in [\ell_i,m_i]}B_y|$ is finite. 
Hence $|(\bigcup_{y \in [\ell_j,m_j]}B_y) \cap (\bigcup_{x \in [\ell_i,m_i]}B_x)|$ is infinite. 
By \cref{claim:2/8}(i) and (ii), $[\ell_i,m_i] \cap [\ell_j,m_j]=\emptyset$.
So the definition of well-order-decomposition implies that $|B_{\ell_i} \cup B_{m_i} \cup B_{\ell_j} \cup B_{m_{j}}|$ is infinite. This contradicts the assumption that the width of $((L, \preccurlyeq), (B_t: t \in L))$ is at most $w$.
\end{proof}

Let $R:= \bigcup_{P \in \C}V(P)$. For any $i \in \{1, \dots, |\C|\}$ and any $v \in R \cap \bigcup_{x \in [\ell_i,m_i]}B_x$, let $A_{i,v} := \{u \in R: \{u,v\} \subseteq B_x$ for some $x \in [\ell_i,m_i]\}$. 
Note that $A_{i, v} \subseteq R \cap \bigcup_{x \in [\ell_i,m_i]}B_x$.

\begin{claim} \label{claim:4/8} $|A_{i,v}|$ is finite for any $i \in \{1, \dots, |\C|\}$ and any $v \in (R \cap \bigcup_{x \in [\ell_i,m_i]}B_x) \setminus B_{m_i}'$.
\end{claim}

\begin{proof}
Suppose for the sake of contradiction that there exist $i \in \{1, \dots, |\C|\}$ and $v \in (R \cap \bigcup_{x \in [\ell_i,m_i]}B_x) \setminus B_{m_i}'$ such that $|A_{i,v}|$ is infinite. 
Since $A_{i, v} \subseteq R$, there exists $P_j \in \C$ such that $|V(P_j) \cap A_{i,v}|$ is infinite.  

Since $A_{i, v} \subseteq \bigcup_{x \in [\ell_i,m_i]}B_x$ and $|V(P_j) \cap A_{i,v}|$ is infinite, $|V(P_j) \cap \bigcup_{x \in [\ell_i,m_i]}B_x|$ is infinite.
Hence $m_{j} = m_{i}$ by \cref{claim:3'/8}.

Since $v \in (\bigcup_{x \in [\ell_i,m_i]}B_x) \setminus B_{m_i}'$, the definition of $B_{m_i}'$ implies that there exists $\alpha \in [\ell_i,m_i)$ with $v \not \in \bigcup_{x \in [\alpha,m_i]}B_x \setminus B_\alpha$.
Since $m_j = m_i$, \cref{claim:1/8} implies that there exists a positive integer $N$ such that $V(P_j(N)) \subseteq \bigcup_{x \in [\alpha,m_i]}B_x$. 
Since $|V(P_j) \cap A_{i,v}|$ is infinite, $|V(P_j(N)) \cap A_{i,v}|$ is infinite. 
So $|A_{i, v} \cap \bigcup_{x \in [\alpha,m_i] \setminus \{\alpha\}}B_x|$ is infinite. By the definition of $A_{i, v}$ and since $v \not \in \bigcup_{x \in [\alpha,m_i]}B_x \setminus B_\alpha$, the definition of well-order-decomposition implies that there exist infinitely many vertices of $A_{i, v}$ contained in $B_\alpha$. 
This contradicts the assumption that the width of $((L, \preccurlyeq), (B_t: t \in L))$ is at most $w$.
\end{proof}

For each $i \in \{1, \dots, |\C|\}$, let $O_i := \bigcup\{V(P): P \in \C, m_P=m_i\}$; note that $V(P_i) \subseteq O_{i}$.
For each $i \in \{1, \dots, |\C|\}$, let $\ell_{i,0} :=\ell_i$. For each $i \in \{1, \dots, |\C|\}$ and for every $j \in \mathbb{N}$, if $\ell_{i,j-1}$ is defined and $\{x \in [\ell_{i,j-1},m_i]: (O_i \cap B_x \cap \bigcup_{y \in [\ell_i,\ell_{i,j-1}]}B_{y}) \setminus B_{m_i}'=\emptyset\} \neq \emptyset$, then let $\ell_{i,j}$ be the least element of $\{x \in [\ell_{i,j-1},m_i]: (O_i \cap B_x \cap \bigcup_{y \in [\ell_i,\ell_{i,j-1}]}B_{y}) \setminus B_{m_i}'=\emptyset\}$.

\begin{claim} \label{claim:5/8} $\ell_{i,j}$ is defined and $|\bigcup_{x \in [\ell_{i,j-1},\ell_{i,j}]}B_x|$ is finite for any $i \in \{1, \dots, |\C|\}$ and any $j \in \mathbb{N}$.
\end{claim}
\begin{proof}
Suppose for the sake of contradiction that there exist $i \in \{1, \dots, |\C|\}$ and $j \in \mathbb{N}$ such that either $\ell_{i,j}$ is undefined, or $\ell_{i,j}$ is defined but $|\bigcup_{x \in [\ell_{i,j-1},\ell_{i,j}]}B_x|$ is infinite; subject to this, choose $j$ minimum.
So $\ell_{i,j-1}$ is defined.
By \cref{claim:1/8}(ii) and \cref{claim:2/8}(i), $P_i$ contains infinitely many vertices in $\bigcup_{x \in [\ell_i,m_i]}B_x$. Since $V(P_{i}) \subseteq O_{i}$, we have that $|O_{i} \cap \bigcup_{x \in [\ell_i,m_i]}B_x|$ is infinite. 
Since $B_{m_i}'$ is finite, $|(O_{i} \cap \bigcup_{x \in [\ell_i,m_i]}B_x) \setminus B_{m_i}'|$ is infinite. 
By the minimality of $j$ and since $\ell_{i, 0} = \ell_{i}$, we have that $\bigcup_{x \in [\ell_{i},\ell_{i,1}]}B_x$, $\bigcup_{x \in [\ell_{i, 1},\ell_{i, 2}]}B_x$, $\dots$, $\bigcup_{x \in [\ell_{i, j - 2},\ell_{i, j - 1}]}B_x$ are finite sets. Hence $|\bigcup_{x \in [\ell_{i},\ell_{i,j-1}]}B_x|$ is finite. 
Since $|(O_{i} \cap \bigcup_{x \in [\ell_i,m_i]}B_x) \setminus B_{m_i}'|$ is infinite and $|\bigcup_{x \in [\ell_{i},\ell_{i,j-1}]}B_x|$ is finite, $|(O_{i} \cap \bigcup_{x \in [\ell_i,m_i]}B_x) \setminus (B_{m_i}' \cup \bigcup_{x \in [\ell_{i},\ell_{i,j-1}]}B_x)|$ is infinite. 
So $\{x \in [\ell_{i,j-1},m_i]: (O_i \cap B_x \cap \bigcup_{y \in [\ell_i,\ell_{i,j-1}]}B_{y}) \setminus B_{m_i}'=\emptyset\} \neq \emptyset$.
Hence $\ell_{i,j}$ is defined.

By the definition of $\ell_{i,j}$, for every $x \in [\ell_{i,j-1},\ell_{i,j})$, we have $(O_i \cap B_x \cap \bigcup_{y \in [\ell_i,\ell_{i,j-1}]}B_{y}) \setminus B_{m_i}' \neq \emptyset$, so $(O_i \cap B_x \cap B_{\ell_{i,j-1}}) \setminus B_{m_i}' \neq \emptyset$ by the definition of well-order-decomposition. 
For every $x \in [\ell_{i,j-1},\ell_{i,j})$, if $u \in R \cap B_{x}$ and $v \in (O_i \cap B_x \cap B_{\ell_{i,j-1}}) \setminus B_{m_i}'$, then $u, v \in B_{x}$, so $u \in A_{i, v}$ by the definition of $A_{i, v}$. 
This shows that for every $x \in [\ell_{i,j-1},\ell_{i,j})$, we have $R \cap B_x \subseteq \bigcup_{v \in O_i \cap B_{\ell_{i,j-1}} \setminus B_{m_i}'}A_{i,v}$.
Hence $R \cap \bigcup_{x \in [\ell_{i,j-1},\ell_{i,j}]}B_x \subseteq B_{\ell_{i, j}} \cup \bigcup_{v \in O_i \cap B_{\ell_{i,j-1}} \setminus B_{m_i}'}A_{i,v}$. By \cref{claim:4/8}, $|A_{i,v}|$ is finite for every $v \in O_i \cap B_{\ell_{i,j-1}} \setminus B_{m_i}'$. 
Since $|O_i \cap B_{\ell_{i,j-1}} \setminus B_{m_i}'| \leq |B_{\ell_{i,j-1}}|$ is finite, $|R \cap \bigcup_{x \in [\ell_{i,j-1},\ell_{i,j}]}B_x|$ is finite.

Let $W$ be a connected component of $G-(B_{\ell_{i,j-1}} \cup B_{\ell_{i,j}} \cup R)$ intersecting $\bigcup_{x \in [\ell_{i,j-1},\ell_{i,j}]}B_x$.
By the maximality of $\C$, the graph $W$ contains no $1$-way infinite path.
Since $G$ is locally finite, $|V(W)|$ is finite.
Since every component of $G$ contains a $1$-way infinite path, $W$ is contained in a component of $G$ intersecting $R$. 
Hence $V(W)$ is adjacent in $G$ to $B_{\ell_{i,j-1}} \cup B_{\ell_{i,j}} \cup R$.
Since $G$ is locally finite and $|R \cap \bigcup_{x \in [\ell_{i,j-1},\ell_{i,j}]}B_x|$ is finite, there are only finitely many such components~$W$.
This shows that $|\bigcup_{x \in [\ell_{i,j-1},\ell_{i,j}]}B_x|$ is finite and completes the proof of \cref{claim:5/8}.
\end{proof}

\begin{claim} \label{claim:6/8}
For any $i \in \{1, \dots, |\C|\}$ and any $x \in [\ell_i,m_i)$, there exists $j \geq 1$ such that $x \preccurlyeq \ell_{i,j}$.
\end{claim}

\begin{proof}
Suppose for the sake of contradiction that there exist $i \in \{1, \dots, |\C|\}$ and $x \in [\ell_i,m_i)$ such that $\ell_{i,j} \prec x$ for every $j \geq 1$.
By \cref{claim:2/8}(iii), $B_{\ell_i} \setminus B_{m_{i}}'$ intersects the infinite subpath of $P_i - B_{m_i}'$. So there exists a positive integer $a$ such that $v_{P_i,a} \in B_{\ell_i} \setminus B_{m_i}'$ and $v_{P_i, a}$ belongs to the infinite subpath of $P_i - B_{m_i}'$. Since $|B_{\ell_i}| \leqslant w + 1$, we can choose a maximum such $a$; that is, $v_{P_i,j} \not \in B_{\ell_i}$ for every $j > a$.
By \cref{claim:1/8}(ii) and \cref{claim:2/8}(i), there exists a positive integer $N$ such that $V(P_i(N)) \subseteq \bigcup_{y \in [\ell_i,m_i]}B_y$. 
Then the maximality of $a$ implies $V(P_i(a)) \subseteq (\bigcup_{y \in [\ell_i,m_i]}B_y) \setminus B_{m_i}'$.

We now show by induction on $d$ that for every $d \geq 0$, we have $\{v_{P_i,j}: a \leq j \leq a + d\} \subseteq (\bigcup_{\alpha \in [\ell_i,\ell_{i,d}]}B_\alpha) \setminus B_{m_i}'$. 
The base case with $d = 0$ immediately follows from the choice of~$a$. 
Now assume that $d \geqslant 1$ and $\{v_{P_i,j}: a \leq j \leq a + d - 1\} \subseteq \bigcup_{\alpha \in [\ell_i,\ell_{i,d - 1}]}B_\alpha \setminus B_{m_i}'$. 
In particular, $v_{P_i, a + d - 1} \in \bigcup_{\alpha \in [\ell_i,\ell_{i,d - 1}]}B_\alpha \setminus B_{m_i}'$.
By the definition of $\ell_{i, d}$, we have $(V(P_i) \setminus B_{m_i}') \cap B_{\ell_{i, d}} \cap \bigcup_{\alpha \in [\ell_i,\ell_{i,d - 1}]}B_\alpha = \emptyset$.
In particular, $v_{P_i,a+d-1} \not \in B_{\ell_{i,d}}$.
Then, since $v_{P_i, a + d - 1}v_{P_i, a + d} \in E(G)$, the definition of well-order-decomposition implies that $v_{P_i, a + d} \in \bigcup_{\alpha \in [\ell_i,\ell_{i,d}]}B_\alpha$.
Since $v_{P_i, a + d} \in V(P_i(a))$, we know $v_{P_i, a + d} \notin B_{m_i}'$.

So we have shown that for every $d \geq 0$, we have $\{v_{P_i,j}: a \leq j \leq a + d\} \subseteq \bigcup_{\alpha \in [\ell_i,\ell_{i,d}]}B_\alpha \setminus B_{m_i}'$.
By \cref{claim:1/8}(ii), there exists an integer $b \geq a$ such that $x \preccurlyeq h_{v_{P_i,b}}$.
Setting $d \coloneqq b - a$, we obtain that $v_{P_{i}, b} \in \bigcup_{\alpha \in [\ell_i,\ell_{i,b - a}]}B_\alpha$. 
So by the definition of $h_{v_{P_i, b}}$, we have $h_{v_{P_i,b}} \preccurlyeq \ell_{i,b-a}$.
Therefore $x \preccurlyeq h_{v_{P_i,b}} \preccurlyeq \ell_{i,b-a} \prec x$, a contradiction. 
 \end{proof}

\begin{claim} \label{claim:7/8} For any $i \in \{1, \dots, |\C|\}$, there exists a path-decomposition of $G[\bigcup_{x \in [\ell_i,m_i]}B_x]$ of width at most $3w + 2$ such that every bag contains $B_{\ell_i} \cup B_{m_i}$.
\end{claim}

\begin{proof}
Fix $i \in \{1, \dots, |\C|\}$.
For every $j \geq 0$, $(B_x: \ell_{i,j} \preccurlyeq x \preccurlyeq \ell_{i,j+1})$ constitutes a well-order-decomposition of $G[\bigcup_{x \in [\ell_{i,j},\ell_{i,j+1}]}B_x]$ of width at most $w$.

By \cref{claim:5/8}, $|V(G[\bigcup_{x \in [\ell_{i,j},\ell_{i,j+1}]}B_x])|$ is finite. Thus there are only finitely many different bags in $(B_x: \ell_{i,j} \preccurlyeq x \preccurlyeq \ell_{i,j+1})$. 
So by removing duplicated bags, we obtain a path-decomposition of $G[\bigcup_{x \in [\ell_{i,j},\ell_{i,j+1}]}B_x]$ of width at most $w$. 
That is, for every $j \geq 0$, there exist a positive integer $n_j$ and a path-decomposition $(D_x: 1 \leq x \leq n_j)$ of $G[\bigcup_{x \in [\ell_{i,j},\ell_{i,j+1}]}B_x]$ of width at most $w$ such that $D_1=B_{\ell_{i,j}}$ and $D_{n_j}=B_{\ell_{i,j+1}}$.

By concatenating these path-decompositions of $G[\bigcup_{x \in [\ell_{i,0},\ell_{i,1}]}B_x]$, $G[\bigcup_{x \in [\ell_{i,1},\ell_{i,2}]}B_x], \dots$, we obtain a path-decomposition $(F_x: x \in \mathbb{N})$ of $G[\bigcup_{j=0}^\infty\bigcup_{x \in [\ell_{i,j},\ell_{i,j+1}]}B_x]$ of width at most~$w$.
By \cref{claim:6/8}, $G[\bigcup_{j=0}^\infty\bigcup_{x \in [\ell_{i,j},\ell_{i,j+1}]}B_x]$ contains $G[(\bigcup_{x \in [\ell_i,m_i]}B_x) \setminus B_{m_i}]$.
By adding $B_{\ell_i} \cup B_{m_i}$ to every bag of $(F_x: x \in \mathbb{N})$, we obtain a path-decomposition of $G[\bigcup_{x \in [\ell_i,m_i]}B_x]$ of width at most $w + |B_{\ell_i}| + |B_{m_i}| \leq 3w + 2$.
\end{proof}

\begin{claim} \label{claim:8/8} $|V(G) \setminus \bigcup_{i=1}^{|\C|}\bigcup_{x \in [\ell_i,m_i]}B_x|$ is finite.
\end{claim}

\begin{proof}
Let $W$ be a connected component of $G-\bigcup_{i=1}^{|\C|}\bigcup_{x \in [\ell_i,m_i]}B_x$. Suppose that $W$ contains a $1$-way infinite path. By the maximality of $\C$, there exists $P_i \in \C$ such that $|W \cap P_i|$ is infinite. This contradicts \cref{claim:3/8}. Thus $W$  contains no $1$-way infinite path. Since $G$ is locally finite, $W$ is finite.

Suppose that $W$ is not adjacent in $G$ to $\bigcup_{i=1}^{|\C|}(B_{\ell_i} \cup B_{m_i})$. 
Then by the definition of well-order-decomposition, $W$ is not adjacent in $G$ to $\bigcup_{i=1}^{|\C|}\bigcup_{x \in [\ell_i,m_i]}B_x$. 
So $W$ is also a connected component of $G$. 
By assumption of this lemma, $W$ contains a $1$-way infinite path, a contradiction.

So every component of $G-\bigcup_{i=1}^{|\C|}\bigcup_{x \in [\ell_i,m_i]}B_x$ is finite and is adjacent in $G$ to $\bigcup_{i=1}^{|\C|}(B_{\ell_i} \cup B_{m_i})$.
Since $|\bigcup_{i=1}^{|\C|}(B_{\ell_i} \cup B_{m_i})| \leqslant 2(w + 1)|\C|$ and $G$ is locally finite, there are only finitely many components of $G-\bigcup_{i=1}^{|\C|}\bigcup_{x \in [\ell_i,m_i]}B_x$.
Since every component of $G-\bigcup_{i=1}^{|\C|}\bigcup_{x \in [\ell_i,m_i]}B_x$ is finite, $G-\bigcup_{i=1}^{|\C|}\bigcup_{x \in [\ell_i,m_i]}B_x$ is finite.
\end{proof}

By \cref{claim:8/8}, we obtain a path-decomposition $(Z_\alpha: \alpha \in \mathbb{N})$ of $G - \bigcup_{i=1}^{|\C|}\bigcup_{x \in [\ell_i,m_i]}B_x$ of width at most $w$ by removing duplicated bags of $((L, \preccurlyeq),(B_t \setminus \bigcup_{i=1}^{|\C|}\bigcup_{x \in [\ell_i,m_i]}B_x: t \in L))$.
By \cref{claim:7/8}, for every $i \in \{1, \dots, |\C|\}$, there exists a path-decomposition $(Z^i_\alpha: \alpha \in \mathbb{N})$ of $G[\bigcup_{x \in [\ell_i,m_i]}B_x]$ of width at most $3w + 2$ such that every bag contains $B_{\ell_i} \cup B_{m_i}$.
Therefore $(Z_\alpha \cup \bigcup_{i=1}^{|\C|}Z^i_\alpha: \alpha \in {\mathbb N})$ is a path-decomposition of $G$ of width at most $w+|\C|(3w+3) \leq w+3(k-1)(w+1)$.
This completes the proof of this lemma. 
\end{proof}

\begin{thm} \label{thm:WodToPd}
For any graph $G$ and integers $k\geq 1$ and $d\geq 0$, if at most $d$ vertices in $G$ have infinite degree, and $G$ contains no $k$ pairwise disjoint $1$-way infinite paths, then 
\[\pw(G) \leqslant (3k-1)\wow(G)  + 3k+d-2.\]
\end{thm}

\begin{proof}
If $\wow(G) = \infty$, then the inequality is immediate. Otherwise  $w \coloneqq \wow(G)$ is finite. 
Let $X$ be the set of vertices in $G$ with infinite degree.
So $|X|\leq d$ and $G - X$ is locally finite.

Let $H$ be the union of the connected components of $G - X$ that contain no $1$-way infinite path.
Since $G - X$ is locally finite, every component of $H$ is finite.
Since $G$ has a well-order-decomposition of width at most $w$, every component $C$ of $H$ has a path-decomposition $(B^C_i: i \in \{1, \dots, n_C\})$ of width at most $w$ for some positive integer $n_C$.
By concatenating these path-decompositions, we obtain a path-decomposition $(B^H_i: i \in {\mathbb N})$ of $H$ of width at most~$w$.

Note that $G - (X \cup V(H))$ is a locally finite graph that contains no $k$ pairwise disjoint 1-way infinite paths, and every connected component of $G - (X \cup V(H))$ contains a $1$-way infinite path. 
By \cref{lem:NoFiniteComp}, there exists a path-decomposition $(B_i: i \in \mathbb{N})$ of $G - (X \cup V(H))$ of width at most $w + 3(k - 1)(w + 1)$. 
Hence $(B_i \cup B^H_i: i \in \mathbb{N})$ is a path-decomposition of $G-X$ of width at most $w + (3k-2)(w + 1)$. 
Therefore $(B_i \cup B^H_i \cup X: i \in \mathbb{N})$ is a path-decomposition of $G$ with width at most $w + (3k-2)(w + 1) + |X| \leq (3k-1)w + 3k+d-2$.
\end{proof}

We are now ready to prove \cref{FinitePathwidthQuantitative}.

\begin{proof}[Proof of \cref{FinitePathwidthQuantitative}]
Let $G$ be a graph such that for some integers $k\geq1$ and $d,c\geq 0$: 
    \begin{itemize}
        \item at most~$d$ vertices in $G$ have infinite degree,
        \item $G$ contains no~$k$ pairwise disjoint 1-way infinite paths, and
        \item every finite subgraph of~$G$ has path-width at most~$c$.
    \end{itemize}
    By \cref{FiniteLinewidth}, $\lw(G) \leqslant c$. 
    By \cref{lwandwow}, $\wow(G) \leq 2c$. 
    By \cref{thm:WodToPd}, $\pw(G) \leq 2c + (3k-2)(2c + 1) + d$.
\end{proof}

\section{Universal Construction: Proof of \texorpdfstring{\cref{thm:universallw}}{Theorem 5}} \label{sec:proof2}

This section constructs a universal graph for $\mathcal{L}_k$ that is `close' to $\mathcal{L}_k$, thus proving \cref{thm:universallw}. We first present an analogous result for well-order-width.

\begin{restatable}{thm}{thmuniversalwow} \label{universalwow}
    For any integer $k \geq 0$, there is a universal graph for $\mathcal{W}_k$ with well-order-width at most $k^2 + 3k$.
\end{restatable}

The proof of \cref{universalwow} adapts a standard method~\citep{DJKW16,BEGHRW,NSW22} in the analysis of path-decompositions of finite graphs to the setting of infinite graphs. Here one takes a path-decomposition of a graph, finds a maximal set of pairwise disjoint bags, deletes these bags, and obtains a subgraph with path-width less than the path-width of the original graph. In particular, \citet{BEGHRW} used this method to prove a result analogous to \cref{universalwow} (or \cref{thm:universallw}) for finite graphs.

\begin{proof} 
We proceed by induction on $k$. 
In the base case $k = 0$, let $U_0$ consist of infinitely many isolated vertices. Note that a graph has well-order-width at most $0$ if and only if it has no edges. Thus $U_0$ has well-order-width $0$ and contains every graph with well-order-width $0$. 
 
Now assume that $k \geq 1$ and $U_{k - 1}$ is a universal graph with well-order-width at most $(k - 1)^{2} + 3(k - 1)$ for $\mathcal{W}_{k-1}$.
Let $(H_i:i\in \mathbb{N}_0)$ be a collection of pairwise disjoint copies of $U_{k-1}$. 
Let $(S_i:i\in\mathbb{N}_0)$ be a collection of pairwise disjoint sets of vertices disjoint from $\bigcup_{i \in \NN_0}V(H_i)$, with $|S_i|=k+1$ for each $i \in \mathbb{N}_0$. As illustrated in \cref{universalgraphpicture}, let $U_k^{-}$ be the graph obtained from $\bigcup_{i \in {\mathbb N}_0}H_i$ by adding $\bigcup_{i \in \NN_0}S_i$ and, for every $i \in {\mathbb N}_0$, adding edges (i) between every vertex in $H_i$ and every vertex in $S_i \cup S_{i+1}$, and (ii) between any two distinct vertices of $S_i$ so that $S_{i}$ is a clique in $U_{k}^{-}$, and (iii) between every vertex in $S_{i}$ and every vertex in $S_{i + 1}$.

    \begin{figure}[ht]
        \centering
        \scalebox{1}{\includegraphics{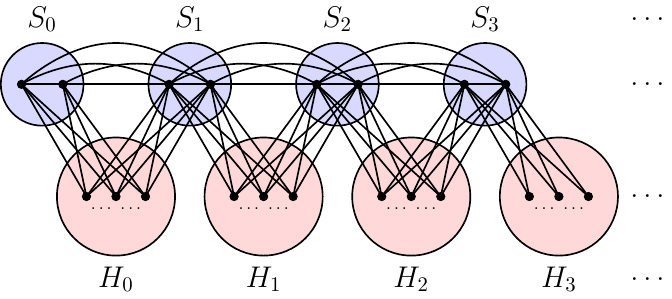}}
        \caption{The construction of $U_{k}^{-}$. In this illustration, $k = 1$, so $|S_{i}| = k + 1 = 2$ for each $i \in \mathbb{N}_{0}$.
        }
        \label{universalgraphpicture}
    \end{figure}

\begin{claim} \label{claim:one} $U_k^{-}$ has well-order-width at most $k^2 + 3k$.
    
\end{claim}
\begin{proof}
Our proof modifies the proof of \cref{lem:unionwow}. By the inductive hypothesis, for each $i \in \mathbb{N}_{0}$, there is a well-order-decomposition $((L_{i}, \preccurlyeq_{i}), (B_{t}^{i} : t \in L_{i}))$ of $H_{i}$ with width at most $(k - 1)^2 + 3(k - 1)$. So for each $i \in \mathbb{N}_{0}$ and each $t \in L_{i}$, we have $|B_{t}^{i}| \leqslant (k - 1)^2 + 3(k - 1) + 1$.
We may and will assume that $L_{0}, L_{1}, \dots$ are pairwise disjoint. Let $L := L_{0} \cup L_{1} \cup \dots$ and $\preccurlyeq$ be the concatenation of $(\preccurlyeq_{0}$, $\preccurlyeq_{1}, \dots)$. That is, for each $t_1, t_2 \in L$, we have $t_1 \preccurlyeq t_2$ if and only if: (i) $t_1 \in L_{i}$ and $t_2 \in L_{j}$ for some $i, j \in \mathbb{N}_{0}$ with $i < j$, or (ii) $t_1 \in L_{i}$, $t_2 \in L_{i}$ and $t_1 \preccurlyeq_{i} t_2$ for some $i \in \mathbb{N}_{0}$. Since $\preccurlyeq_{i}$ is a well-order of $L_{i}$ for each $i \in \mathbb{N}_0$, we have that $\preccurlyeq$ is a well-order of $L$. For each $i \in \mathbb{N}_{0}$ and for each $t \in L_{i}$, let $B_{t} :=  B_{t}^{i} \cup S_{i} \cup S_{i + 1}$. Observe that $((L, \preccurlyeq), (B_{t} : t \in L))$ is a well-order-decomposition of $U_k^{-}$. Since $|B_{t}| \leqslant (k - 1)^2 + 3(k - 1) + 1 + 2(k + 1) = k^2 + 3k + 1$ for each $t \in L$, the width of $((L, \preccurlyeq), (B_{t} : t \in L))$ is at most $k^2 + 3k$. Therefore, $U_k^{-}$ has well-order-width at most $k^2 + 3k$.
\end{proof}

\begin{claim} \label{claim:two} $U_k^{-}$ contains every connected graph with well-order-width at most $k$.
\end{claim}

\begin{proof}
Let $G$ be a connected graph with well-order-width at most $k$. Our goal is to show that $G$ is isomorphic to a subgraph of $U_k^{-}$.

Let $((L, \preccurlyeq), (B_{t} : t \in L))$ be a well-order-decomposition of $G$ with width at most $k$.
So $|B_t| \leq k + 1$ for each $t \in L$. We may and will assume that $B_t \neq \emptyset$ for every $t \in L$.

We now define a sequence $S = (x_0, x_1, \dots)$ of distinct elements of $L$, where $S$ might be finite or infinite. First, let $x_0$ be the least element of $L$, which exists since $\preccurlyeq$ is a well-order of $L$. Now assume that $x_0, x_1, \dots, x_{i - 1}$ are defined for some $i \geqslant 1$. If $\{x\in L: B_x \cap \bigcup_{j=0}^{i-1}B_{x_j}=\emptyset\} = \emptyset$, then we complete the definition of $S$ by setting $S := (x_0, x_1, \dots, x_{i - 1})$. Otherwise, define $x_i$ to be the least element of $\{x\in L: B_x \cap \bigcup_{j=0}^{i-1}B_{x_j}=\emptyset\}$, which exists since $\preccurlyeq$ is a well-order of $L$. 

Observe that $x_{0} \prec x_{1} \prec x_{2} \prec \dotsb$. Moreover, $B_{x_0}$, $B_{x_1}$, $B_{x_{2}}, \dots$ are pairwise disjoint. 
Let $I$ be the set of integers $i$ such that $x_i$ is defined; so $I = \{0, 1, \dots, q\}$ for some $q \in \mathbb{N}_{0}$ or $I = \mathbb{N}_{0}$.
For each $i \in I$, if $i + 1 \in I$, then let $A_{i} \coloneqq \set{y \in L : x_i \preccurlyeq y \prec x_{i+1}}$; otherwise let $A_{i} \coloneqq \set{y \in L : x_i \preccurlyeq y}$. 
For each $i \in I$, we have $x_{i} \in A_{i}$. 
For each $i \in I$, let $G_i \coloneqq G[\bigcup\set{B_{y} : y \in A_i}]$. 

We first show that every vertex of $G$ is in some $G_i$. If $I$ is finite, then $(A_i \colon i \in I)$ is a partition of $L$ and so the result follows. Thus we may and will assume that $I = \NN_0$. It suffices to show that, for every vertex $v$ of $G$, there is some $i \in \NN_0$ and $t \in L$ such that $t \in A_i$ and $v \in B_t$. We will do this by induction on $\dist_G(v, B_{x_0})$ (this is finite for every vertex since $G$ is connected). If $\dist_G(v, B_{x_0}) = 0$, then $v \in B_{x_0}$ and so we may take $t = x_0$ and $i = 0$. Otherwise, there is some edge $vv' \in E(G)$ such that $\dist_G(v', B_{x_0}) < \dist_G(v, B_{x_0})$. By the induction hypothesis applied to $v'$, there is some $i \in \NN_0$ and $t' \in L$ such that $t' \in A_i$ and $v' \in B_{t'}$. Since $vv'$ is an edge, there is some $t \in L$ such that $v, v' \in B_t$. If $t \in \bigcup_{j = 0}^{i + 1} A_j$, then we are done. Otherwise $t \succcurlyeq x_{i + 2}$. Since $t' \in A_i$, we have $t' \prec x_{i + 1} \prec x_{i + 2} \preccurlyeq t$. But $v' \in B_{t'} \cap B_t$ and so $v' \in B_{x_{i + 1}} \cap B_{x_{i + 2}}$ which contradicts the disjointness of $B_{x_{i + 1}}$ and $B_{x_{i + 2}}$, as required.

We next show that $(G_{i}-(B_{x_i}\cup B_{x_{i + 1}}): i \in I)$ is a partition of $G-\bigcup_{i \in I}B_{x_i}$ into induced subgraphs. By the previous paragraph, it suffices to show that no vertex is in two different $G_i - (B_{x_i} \cup B_{x_{i + 1}})$. Suppose, towards a contradiction, that this is not the case. Then there is a vertex $v \in V(G)-\bigcup_{i \in I}B_{x_i}$, some $i, j \in I$ with $i < j$, and $t \in A_i$ and $t' \in A_j$ such that $v \in B_t \cap B_{t'}$. By the definition of $A_i$ and $A_j$, $t \prec x_{i + 1} \preccurlyeq x_j \preccurlyeq t'$. This implies that $v \in B_{x_{i + 1}}$ which contradicts $v \in V(G)-\bigcup_{i \in I}B_{x_i}$, as required.

For every $v \in V(G) \setminus \bigcup_{i \in I}B_{x_i}$, let $\mathdefn{i_v} \in I$ be the unique integer such that $v \in V(G_{i_v})$. 

We now show that, for all $i \in I$ and $y \in A_i$, $B_y \cap B_{x_i} \neq \emptyset$. If $i + 1 \in I$, then $y \prec x_{i + 1}$ and so, by the definition of $x_{i + 1}$, $B_y \cap \bigcup_{j = 0}^{i} B_{x_j} \neq \emptyset$. Otherwise, $x_i$ is the last element of $S$, so $B_x \cap \bigcup_{j = 0}^{i} B_{x_j} \neq \emptyset$ for every $x \in L$. In either case, there is some $j \leq i$ such that $B_{x_j} \cap B_y \neq \emptyset$. Since $x_j \preccurlyeq x_i \preccurlyeq y$ (the latter inequality follows from $y \in A_i$), $B_{x_i} \cap B_y \neq \emptyset$, as desired. 

For each $i \in I$, let $\preccurlyeq_i$ be the restriction of $\preccurlyeq$ to $A_{i}$. Since $A_i$ is an interval of $L$, $((A_{i}, \preccurlyeq_i),\allowbreak (B_{y} : y \in A_i))$ is a well-order-decomposition of $G_{i}$ with width at most $k$. The previous paragraph implies that every bag intersects $B_{x_i}$. Hence $((A_{i}, \preccurlyeq_i),\allowbreak (B_{y} \setminus B_{x_i} : y \in A_i))$ is a well-order-decomposition of $G_i - B_{x_i}$ with width at most $k - 1$. So $\wow(G_i - B_{x_i}) \leqslant k - 1$. 

Hence for every $i \in I$, there exists an injective homomorphism $\phi_i$ from $G_i-B_{x_i}$ to $H_i$ by the inductive hypothesis.
Since $|B_{x_i}| \leq k+1=|S_i|$ for every $i \in I$ and $B_{x_0}$, $B_{x_1}$, $B_{x_{2}}, \dots$ are pairwise disjoint, there exists an injection $\phi$ from $\bigcup_{i \in I}B_{x_i}$ to $\bigcup_{i \in I}S_i$ such that $\phi$ maps vertices in $B_{x_i}$ to $S_i$ for every $i \in I$.
We extend $\phi$ to an injection from $V(G)$ to $V(U_k^{-})$ by further defining $\phi(v) :=\phi_{i_v}(v)$ for every $v \in V(G) \setminus \bigcup_{i \in I}B_{x_i}$.

We now show that $\phi$ is an injective homomorphism.
Let $uv \in E(G)$ be an arbitrary edge. 
It suffices to show $\phi(u)\phi(v) \in E(U_k^{-})$. 
Since $uv \in E(G)$, there exists $t \in L$ such that $u, v \in B_{t}$.

First, suppose that $\{u,v\} \cap \bigcup_{i \in I}B_{x_i} = \emptyset$. 
Since $u, v \in B_{t}$ and $(G_{i}-(B_{x_i}\cup B_{x_{i + 1}}): i \in I)$ is a partition of $G-\bigcup_{i \in I}B_{x_i}$ into induced subgraphs, we have $i_{u} = i_{v}$ and $uv \in E(G_{i_u})=E(G_{i_v})$. Therefore $\phi(u)=\phi_{i_u}(u)$ and $\phi(v)=\phi_{i_u}(v)$ are the endpoints of an edge of $H_{i_u} \subseteq U_k^{-}$ since $\phi_{i_u}$ is an injective homomorphism.

Now assume that at least one of $u$ and $v$ belongs to $\bigcup_{i \in I}B_{x_i}$. 
Without loss of generality, we may assume $u \in \bigcup_{i \in I}B_{x_i}$. 
Recall that $B_{x_0}$, $B_{x_1}$, $B_{x_{2}}$, \ldots are pairwise disjoint. So there exists a unique $j \in I$ such that $u \in B_{x_j}$.
Hence $\phi(u) \in S_j$. Since $S_j$ is a clique in $U_k^{-}$, we may assume $\phi(v) \not \in S_j$, for otherwise we are done.
In particular, $v \not \in B_{x_j}$.

\textbf{Case 1.}  $j \geqslant 1$ and $t \preccurlyeq x_{j - 1}$:

Since $u \in B_{t} \cap B_{x_{j}}$ and $t \preccurlyeq x_{j - 1} \preccurlyeq x_{j}$, we have $u \in B_{x_{j}} \cap B_{x_{j - 1}}$, implying $B_{x_{j}} \cap B_{x_{j - 1}} \neq \emptyset$. This contradicts the choice of $x_{j}$, so Case $1$ does not occur.

\textbf{Case 2.} $j + 1 \in I$ and $x_{j + 1} \preccurlyeq t$:

Since $u \in B_{t} \cap B_{x_{j}}$ and $x_{j} \preccurlyeq x_{j + 1} \preccurlyeq t$, we have $u \in B_{x_{j}} \cap B_{x_{j + 1}}$, implying $B_{x_{j}} \cap B_{x_{j + 1}} \neq \emptyset$. This contradicts the choice of $x_{j + 1}$, so Case $2$ does not occur.

\textbf{Case 3.} ($j = 0$ or $x_{j - 1} \prec t$) and ($j + 1 \notin I$ or $t \prec x_{j + 1}$):

For the sake of convenience, define $G_{- 1}$ and $H_{ - 1}$ to be the graphs with no vertices and $B_{x_{-1}} := \emptyset$. 
By the construction of $(G_{i} : i \in I)$, the assumption of Case $3$ implies that  $v \in V(G_{j - 1} \cup G_j) \setminus B_{x_j}$.  

Suppose that $v \in B_{x_{j - 1}} \cup B_{x_{j + 1}}$. Then $\phi(v) \in S_{j - 1} \cup S_{j + 1}$. Since every vertex in $S_{j}$ is adjacent to every vertex in $S_{j - 1} \cup S_{j + 1}$, we have $\phi(u)\phi(v) \in E(U_k^{-})$. 

Now assume that $v \notin B_{x_{j - 1}} \cup B_{x_{j + 1}}$. So $v \in V(G_{j - 1} \cup G_j) \setminus (B_{x_{j - 1}} \cup B_{x_j} \cup B_{x_{j + 1}})$. Hence, $\phi(v) \in V(H_{j-1}) \cup V(H_j)$.
Since every vertex in $S_j$ is adjacent to every vertex in $V(H_{j-1}) \cup V(H_j)$, we have $\phi(u)\phi(v) \in E(U_k^{-})$. Therefore, $\phi$ is an injective homomorphism from $G$ to~$U_k^{-}$, and $G$ is isomorphic to a subgraph of $U_{k}^{-}$, as desired.
\end{proof} 

Let $U_{k}$ be the disjoint union of infinitely many disjoint copies of $U_{k}^{-}$. By \cref{claim:one} and \cref{lem:unionwow}, $U_{k}$ has well-order-width at most $k^2 + 3k$. By \cref{claim:two}, $U_{k}$ contains every graph in $\mathcal{W}_k$. This completes the proof.
\end{proof}

We finish this section by proving \cref{thm:universallw} from \cref{sec:intro}.

\universallw*

\begin{proof}
By \cref{lwandwow}, the well-order-width of every graph in $\mathcal{L}_{k}$ is at most $2k$. So $\mathcal{L}_{k} \subseteq \mathcal{W}_{2k}$. By \cref{universalwow}, there is a universal graph $U$ for $\mathcal{L}_{k}$ with well-order-width at most $(2k)^2 + 3(2k) = 4k^2 + 6k$. By \textcolor{blue!60!black}{(}\ref{observation}\textcolor{blue!60!black}{)}, $\lw(U) \leqslant \wow(U) \leqslant 4k^2 + 6k$. 
\end{proof}

\section{Universal Lower Bound: Proof of \texorpdfstring{\cref{thm:lowerbound}}{Theorem 6}} \label{sec:proof3}

This section proves~\cref{thm:lowerbound}, which says that every universal graph for $\mathcal{P}_k$ has line-width at least $k+1$. We start with the following lemma.  

\begin{lem} \label{lem:countableline} Let $G$ be a graph with line-width at most $k$. Then there exists a line-decomposition $((L, \preccurlyeq), (B_{t} : t \in L))$ of $G$ with width at most $k$ such that $L$ is countable and $B_{a} \neq B_{b}$ for all distinct $a, b \in L$. 
\end{lem}

\begin{proof}
Let $((L', \preccurlyeq'), (B_{t} : t \in L'))$ be a line-decomposition of $G$ with width at most $k$.  
For any $s,t \in L'$, define \defn{$s \equiv t$} if and only if $B_s = B_t$. Choose one representative from each equivalence class of $\equiv$, and let the resulting set be $L$. Let $\preccurlyeq$ be the restriction of $\preccurlyeq'$ to $L$. Observe that $((L, \preccurlyeq), (B_{t} : t \in L))$ is also a line-decomposition of $G$ with width at most $k$. By construction, $B_{a} \neq B_{b}$ for all distinct $a, b \in L$. 
Since $G$ has countably many vertices, there are only countably many subsets of $V(G)$ with size at most $k+1$.
Since $|B_{t}| \leqslant k + 1$ for each $t \in L$, we have that $L$ is countable.  
\end{proof}

To prove \cref{thm:lowerbound}, we employ the following auxiliary definitions. For a positive integer $k$, an infinite sequence $\mathbf{s}:= (s_n)_{n \in \mathbb{N}}$ is \defn{$k$-feasible} if:
    \begin{itemize}
    \item $s_i = i$ for all $i \in \{1, \dots, k\}$, and
    \item
    $s_i \in \{1, \dots, i + k - 1\} \setminus \{s_1, \dots, s_{i-1}\}$ for all $i > k$.
    \end{itemize}

    Note that for $i > k$, there are $(i + k - 1) - (i - 1) = k$ possibilities for $s_{i}$. Let \defn{$B_1^{\mathbf{s}}$}~$:= \{1, \dots, k + 1\}$, and for each $i \geq 2$, define \defn{$B_i^{\mathbf{s}}$} $:=(B_{i-1}^{\mathbf{s}} \cup \{k+i\}) \setminus \{s_{i-1}\}$. 
    Note that for each $i \geqslant 1$, we have $|B_{i}^{\mathbf{s}}| = k + 1$ and $\{s_{i}\} = B_{i}^{\mathbf{s}} \setminus B_{i + 1}^{\mathbf{s}}$. See \cref{examplefeasible}. Observe that every integer in $\mathbb{N} \setminus \{1\}$ is in at least two sets of $(B_{1}^{\mathbf{s}}, B_{2}^{\mathbf{s}}, \dots)$.

    \begin{figure}[ht]
        \centering
        \scalebox{0.8}{\includegraphics{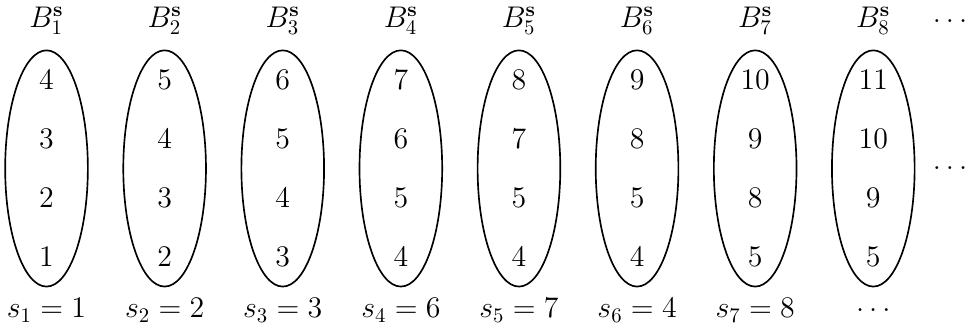}}
        \caption{An example of the beginning of a $k$-feasible sequence $\mathbf{s} = (s_n)_{n \in \mathbb{N}}$ with $k = 3$ and the beginning of the associated sequence of sets $(B_{1}^{\mathbf{s}}, B_{2}^{\mathbf{s}}, \dots)$.}
        \label{examplefeasible}
    \end{figure}

    Given a $k$-feasible sequence $\mathbf{s}:=(s_n)_{n \in \mathbb{N}}$, let \defn{$G_\mathbf{s}$} be a graph with vertex set $\mathbb{N}$ such that two vertices $u$ and $v$ of $G_\mathbf{s}$ are adjacent in $G_\mathbf{s}$ if and only if there exists $i \geqslant 1$ such that $u, v \in B_{i}^{\mathbf{s}}$. Note that $(B_i^{\mathbf{s}} : i \in \mathbb{N})$ is a path-decomposition of~$G_\mathbf{s}$ with width $k$ and $G_\mathbf{s}[B_{i}^{\mathbf{s}}]$ is a $(k + 1)$-clique for every $i \in \mathbb{N}$. Hence $G_\mathbf{s}$ has path-width $k$.

    \begin{lem} \label{lem:notisomorphic}
        For every integer $k \geqslant 1$ and every pair of distinct $k$-feasible sequences $\mathbf{s}:=(s_n)_{n \in \mathbb{N}}$ and $\mathbf{t}:=(t_n)_{n \in \mathbb{N}}$, the graphs $G_\mathbf{s}$ and $G_{\mathbf{t}}$ are not isomorphic.
    \end{lem}

    \begin{proof}
    Suppose for the sake of contradiction that there is an isomorphism $\phi: \mathbb{N} \to \mathbb{N}$ from $G_\mathbf{s}$ to $G_\mathbf{t}$. 
    
    We first show $\phi(n) = n$ for all $n \in \mathbb{N}$ by strong induction on $n$.  
    For the base cases, if $n \in \{1, \dots, k\}$, then $n$ is the unique vertex of degree $k$ in both $G_\mathbf{s} - \{1, \dots, n - 1\}$ and $G_\mathbf{t} - \{1, \dots, n - 1\}$. Thus $\phi(i) = i$ for each $i \in \{1, \dots, k\}$. Now fix $n \geqslant k + 1$ and assume that $\phi(i) = i$ for all $i \in \{1, \dots, n - 1\}$.  Note that $n$ is the unique vertex not in $\{1, \dots, n-1\}$, which has exactly $k$ neighbours in $\{1, \dots, n - 1\}$ (in both $G_\mathbf{s}$ and $G_\mathbf{t}$).  Hence $\phi(n) = n$, as required. 

    Since $\mathbf{s} \neq \mathbf{t}$, there exists the smallest $m \in \mathbb{N}$ such that $s_m \neq t_m$.
    By the definition of a $k$-feasible sequence, $s_i = t_i = i$ for each $i \in \{1, \dots, k\}$.
    So $m \geqslant k+1$.
    Note that $s_{m}$ is the unique vertex of degree $k$ in $G_{\mathbf{s}} - \{s_{1}, \dots, s_{m - 1}\}$, and $t_{m}$ is the unique vertex of degree $k$ in $G_{\mathbf{t}} - \{t_{1}, \dots, t_{m - 1}\}$. 
    Since $\phi$ is the identity function, $s_{m} = t_{m}$, a contradiction.     
    \end{proof}

    For every integer $k \geqslant 1$, let \defn{$\mathcal{S}_k$} be the set of isomorphism classes of $\{G_\mathbf{s} : \mathbf{s} \text{ is } k\text{-feasible}\}$.

    \begin{lem} \label{lem:uncountable} 
        For every integer $k \geqslant 2$, the set $\mathcal{S}_k$ is uncountable.
    \end{lem}

    \begin{proof} 
        Recall that for every $k$-feasible sequence $(s_n)_{n \in \mathbb{N}}$ and all $i > k$, there are $k$ possibilities for $s_{i}$. Since $k \geqslant 2$, the set of $k$-feasible sequences is uncountable. By \cref{lem:notisomorphic}, $\mathcal{S}_k$ is uncountable.
    \end{proof}

 We are now ready to prove \cref{thm:lowerbound}.

\thmlowerbound*
\begin{proof}

    Let $U$ be a graph that contains every graph with path-width at most $k$. Suppose for the sake of contradiction that $\lw(U) \leqslant k$. By \cref{lem:countableline}, there exists a line-decomposition $((L, \preccurlyeq), (U_{t} : t \in L))$ of $U$ with width at most $k$, where $U_{a} \neq U_{b}$ for all distinct $a, b \in L$ and $L$ is countable. 
    
    Let $\mathbf{s}$ be a $k$-feasible sequence. Recall that $(B_i^{\mathbf{s}} : i \in \mathbb{N})$ is a path-decomposition of $G_\mathbf{s}$ with width $k$, and $G_{\mathbf{s}}$ has path-width $k$. Since $U$ contains $G_\mathbf{s}$, there is an injective homomorphism $\psi_{\mathbf{s}}$ from $G_\mathbf{s}$ to~$U$. Let $i \in \mathbb{N}$.  Since $B_i^{\mathbf{s}}$ is a finite clique in $G_\mathbf{s}$, $\psi_{\mathbf{s}}(B_i^{\mathbf{s}}) \subseteq U_{\ell_i^{\mathbf{s}}}$ for some $\ell_i^{\mathbf{s}} \in L$.  Since $|B_i^{\mathbf{s}}| = k + 1$, we have $\psi_{\mathbf{s}}(B_i^{\mathbf{s}})=U_{\ell_i^{\mathbf{s}}}$ and $|\psi_{\mathbf{s}}(B_i^{\mathbf{s}})| = |U_{\ell_i^{\mathbf{s}}}| = k + 1$.  Note that $\ell_i^{\mathbf{s}}$ is unique since  $U_a \neq U_b$ for all distinct $a,b \in L$. Since $B_{1}^{\mathbf{s}}, B_{2}^{\mathbf{s}}, \dots$ are distinct, we have that $\ell_1^{\mathbf{s}}, \ell_2^{\mathbf{s}}, \dots$ are distinct. 
    
    We now show that for each $i \in \mathbb{N}$, either $\ell_i^{\mathbf{s}} \preccurlyeq \ell_{i+1}^{\mathbf{s}} \preccurlyeq \ell_{i+2}^{\mathbf{s}}$ or $\ell_{i+2}^{\mathbf{s}} \preccurlyeq \ell_{i+1}^{\mathbf{s}} \preccurlyeq \ell_i^{\mathbf{s}}$. Suppose for the sake of contradiction that (i) $\ell_{i+2}^{\mathbf{s}} \preccurlyeq \ell_i^{\mathbf{s}} \preccurlyeq \ell_{i+1}^{\mathbf{s}}$ or (ii) $\ell_{i + 1}^{\mathbf{s}} \preccurlyeq \ell_i^{\mathbf{s}} \preccurlyeq \ell_{i + 2}^{\mathbf{s}}$ or (iii) $\ell_{i}^{\mathbf{s}} \preccurlyeq \ell_{i + 2}^{\mathbf{s}} \preccurlyeq \ell_{i + 1}^{\mathbf{s}}$ or (iv) $\ell_{i + 1}^{\mathbf{s}} \preccurlyeq \ell_{i + 2}^{\mathbf{s}} \preccurlyeq \ell_{i}^{\mathbf{s}}$. If (i) or (ii) hold, then $B_{i + 2}^{\mathbf{s}} \cap B_{i + 1}^{\mathbf{s}} \subseteq B_{i}^{\mathbf{s}}$, and if (iii) or (iv) hold, then $B_{i}^{\mathbf{s}} \cap B_{i+1}^{\mathbf{s}} \subseteq B_{i+2}^{\mathbf{s}}$.  However, both possibilities contradict the fact that $\mathbf{s}$ is a $k$-feasible sequence. 
    
    So for each $i \in \mathbb{N}$, either $\ell_i^{\mathbf{s}} \preccurlyeq \ell_{i+1}^{\mathbf{s}} \preccurlyeq \ell_{i+2}^{\mathbf{s}}$ or $\ell_{i+2}^{\mathbf{s}} \preccurlyeq \ell_{i+1}^{\mathbf{s}} \preccurlyeq \ell_i^{\mathbf{s}}$. 
    Hence either $\ell_1^{\mathbf{s}} \preccurlyeq \ell_2^{\mathbf{s}} \preccurlyeq \dotsb$ or $\ell_1^{\mathbf{s}} \succcurlyeq \ell_2^{\mathbf{s}} \succcurlyeq \dotsb$.  
    
    Let $X:=\{x \in V(U) : \text{there exists exactly one element $t \in L$ such that $x \in U_t$}\}.$ 
    Let $L':=\{t \in L : U_t \cap X = \emptyset\}$. Recall that every integer in $\mathbb{N} \setminus \{1\}$ is in at least two bags of the path-decomposition $(B_i^{\mathbf{s}} : i \in \mathbb{N})$ of $G_{\mathbf{s}}$. Hence $\psi_{\mathbf{s}}(V(G_\mathbf{s})) \subseteq (V(U) \setminus X) \cup U_{\ell_1^{\mathbf{s}}}$, and $\ell_i^{\mathbf{s}} \in L'$ for every $i \geq 2$.
    
    We now show that for every $t \in L$, there exists at most one element $t^<$ of $L'$ such that $t^< \prec t$, $|U_t \cap U_{t^<}|=k$, and $|U_{t^<}| = k + 1$. Suppose for the sake of contradiction that there are two such elements $t_{1}^<, t_{2}^< \in L'$ with $t_{1}^< \prec t_{2}^< \prec t$. Then $|U_t \cap U_{t_{1}^<}| = |U_t \cap U_{t_{2}^<}| = k$ and $|U_{t_{1}^<}| = |U_{t_{2}^<}| = k + 1$. Since $((L, \preccurlyeq), (U_{t} : t \in L))$ is a line-decomposition, $U_t \cap U_{t_{1}^<} \subseteq U_{t_{2}^<}$. Let $q$ be the only element of $U_{t_{2}^<} \setminus (U_t \cap U_{t_{1}^<})$, so $U_{t_{2}^<} = \{q\} \cup (U_t \cap U_{t_{1}^<})$. Then $q \notin U_t \cup U_{t_{1}^<}$, for otherwise $U_t=U_{t_{2}^<}$ or $U_{t_{1}^<}=U_{t_{2}^<}$, contradicting $U_{a} \neq U_{b}$ for all distinct $a, b \in L$. Since $t_{2}^{<} \in L'$, we have $q \notin X$. Hence there exists $t_{3} \in L \setminus \{t_{2}^<\}$ such that $q \in U_{t_{3}}$. Since $q \in U_{t_{2}^<}$ and $q \notin U_t \cup U_{t_{1}^<}$, we have $t_{1}^< \prec t_{3} \prec t$. Then $U_t \cap U_{t_{1}^<} \subseteq U_{t_{3}}$. So $U_{t_3} = \{q\} \cup (U_t \cap U_{t_{1}^<})$. Hence $U_{t_3} = U_{t_{2}^<}$. However, $U_{a} \neq U_{b}$ for all distinct $a, b \in L$, a contradiction.
    
    So for every $t \in L$, there exists at most one element $t^<$ of $L'$ such that $t^< \prec t$, $|U_t \cap U_{t^<}|=k$, and $|U_{t^<}| = k + 1$. Similarly, for every $t \in L$, there exists at most one element $t^>$ of $L'$ such that $t \prec t^>$, $|U_t \cap U_{t^>}| = k$, and $|U_{t^>}| = k + 1$.
    Thus, if $\ell_1^{\mathbf{s}} \succcurlyeq \ell_2^{\mathbf{s}}$, then $\ell_1^{\mathbf{s}} \succcurlyeq \ell_2^{\mathbf{s}} \succcurlyeq \dots $ and $\ell_{i+1}^{\mathbf{s}}=\ell_i^<$ for every $i \in \mathbb{N}$.
    Similarly, if $\ell_1^{\mathbf{s}} \preccurlyeq \ell_2^{\mathbf{s}}$, then $\ell_1^{\mathbf{s}} \preccurlyeq \ell_2^{\mathbf{s}} \preccurlyeq \dots $ and $\ell_{i+1}^{\mathbf{s}}=\ell_i^>$ for every $i \in \mathbb{N}$.
Hence, the sequence $(\ell_1^{\mathbf{s}}, \ell_2^{\mathbf{s}}, \ell_3^{\mathbf{s}}, \dots)$ is uniquely determined by $\ell_1^{\mathbf{s}}$ and $\ell_2^{\mathbf{s}}$.

We now show that $\psi_{\mathbf{s}}$ is an isomorphism from $G_{\mathbf{s}}$ to $U[U_{\ell_{1}^{\mathbf{s}}} \cup U_{\ell_{2}^{\mathbf{s}}} \cup \dots]$. Recall that for each $i \in \mathbb{N}$, the set $B_i^{\mathbf{s}}$ is a $(k + 1)$-clique in $G_\mathbf{s}$, and $\psi_{\mathbf{s}}(B_i^{\mathbf{s}})=U_{\ell_i^{\mathbf{s}}}$, and $|\psi_{\mathbf{s}}(B_i^{\mathbf{s}})| = |U_{\ell_i^{\mathbf{s}}}| = k + 1$. 
This implies that the image of $\psi_{\mathbf{s}}$ equals $U_{\ell_{1}^{\mathbf{s}}} \cup U_{\ell_{2}^{\mathbf{s}}} \cup \dotsb$.
Since $\psi_{\mathbf{s}}$ is an injective homomorphism from $G_{\mathbf{s}}$ to $U$, we know that $\psi_{\mathbf{s}}$ is a bijection from $V(G_{\mathbf{s}})$ to $U_{\ell_{1}^{\mathbf{s}}} \cup U_{\ell_{2}^{\mathbf{s}}} \cup \dots$ and is an injective homomorphism from $G_{\mathbf{s}}$ to $U[U_{\ell_{1}^{\mathbf{s}}} \cup U_{\ell_{2}^{\mathbf{s}}} \cup \dots]$.
Suppose for the sake of contradiction that $\psi_{\mathbf{s}}$ is not an isomorphism from $G_{\mathbf{s}}$ to $U[U_{\ell_{1}^{\mathbf{s}}} \cup U_{\ell_{2}^{\mathbf{s}}} \cup \dots]$.
Then $\psi_{\mathbf{s}}$ maps two non-adjacent vertices in $G_{\mathbf{s}}$ to two adjacent vertices $a$ and $b$ in $U[U_{\ell_{1}^{\mathbf{s}}} \cup U_{\ell_{2}^{\mathbf{s}}} \cup \dots]$.
Let $i, j \in \mathbb{N}$ with $a \in U_{\ell_{i}^{\mathbf{s}}}$ and $b \in U_{\ell_{j}^{\mathbf{s}}}$ such that $|j - i|$ is minimum.
Since $\psi_{\mathbf{s}}(B_i^{\mathbf{s}})=U_{\ell_i^{\mathbf{s}}}$ and $\psi_{\mathbf{s}}(B_j^{\mathbf{s}})=U_{\ell_j^{\mathbf{s}}}$, we have $\psi_{\mathbf{s}}^{-1}(a) \in B_i^{\mathbf{s}}$ and $\psi_{\mathbf{s}}^{-1}(b) \in B_j^{\mathbf{s}}$.
Since $\psi_{\mathbf{s}}^{-1}(a)\psi_{\mathbf{s}}^{-1}(b) \notin E(G_{\mathbf{s}})$ and $B_i^{\mathbf{s}}$ is a clique, we know $i \neq j$.
Since $a$ and $b$ appear in a common bag of the line-decomposition $((L, \preccurlyeq), (U_{t} : t \in L))$ and the sequence $(\ell_i^{\mathbf{s}}: i \in \mathbb{N})$ is monotone, we may and will assume that $j = i + 1$. 
By the minimality of $|j - i|$, we have $\psi_{\mathbf{s}}^{-1}(a) \in B_{i}^{\mathbf{s}} \setminus B_{i + 1}^{\mathbf{s}}$ and $\psi_{\mathbf{s}}^{-1}(b) \in B_{i + 1}^{\mathbf{s}} \setminus B_{i}^{\mathbf{s}}$. By construction of $G_{\mathbf{s}}$, we have that $\psi_{\mathbf{s}}^{-1}(b)$ is adjacent to each vertex of $B_{i + 1}^{\mathbf{s}} \setminus \{\psi_{\mathbf{s}}^{-1}(b)\} = B_{i}^{\mathbf{s}} \setminus \{\psi_{\mathbf{s}}^{-1}(a)\}$. Since $ab \in E(U)$, we have that $\psi_{\mathbf{s}}(B_{i}^{\mathbf{s}}) \cup \{b\}$ is a $(k + 2)$-clique in $U$. This contradicts the assumption that $\lw(U) \leqslant k$.

Hence $\psi_{\mathbf{s}}$ is an isomorphism from $G_{\mathbf{s}}$ to $U[U_{\ell_{1}^{\mathbf{s}}} \cup U_{\ell_{2}^{\mathbf{s}}} \cup \cdots]$. 
By \cref{lem:notisomorphic}, for any two distinct $k$-feasible sequences $\mathbf{s}$ and $\mathbf{t}$, $G_{\mathbf{s}}$ and $G_{\mathbf{t}}$ are non-isomorphic, so $(\ell_1^{\mathbf{s}}, \ell_2^{\mathbf{s}}, \ell_3^{\mathbf{s}}, \dots) \neq (\ell_1^{\mathbf{t}}, \ell_2^{\mathbf{t}}, \ell_3^{\mathbf{t}}, \dots)$. 
Recall that the sequence $(\ell_1^{\mathbf{s}}, \ell_2^{\mathbf{s}}, \ell_3^{\mathbf{s}}, \dots)$ is uniquely determined by $\ell_1^{\mathbf{s}}$ and $\ell_2^{\mathbf{s}}$. So for any two distinct $k$-feasible sequences $\mathbf{s}$ and $\mathbf{t}$, we have $(\ell_1^{\mathbf{s}}, \ell_2^{\mathbf{s}}) \neq (\ell_1^{\mathbf{t}}, \ell_2^{\mathbf{t}})$. 
Hence, the cardinality of $\mathcal{S}_k$ is at most the cardinality of the set of ordered pairs of $L$. By \cref{lem:uncountable}, $\mathcal{S}_k$ is uncountable. So $L$ is uncountable, a contradiction.
Therefore, $\lw(U) \geqslant k + 1$, as desired. 
\end{proof}

\subsection*{Acknowledgements}

This research was initiated at the workshop, \href{https://www.matrix-inst.org.au/events/global-structure-and-geometry-of-graphs/}{\it Global Structure and Geometry of Graphs}, held in April 2026 at MATRIX in Creswick, Australia.
This paper was partially written when the fifth author visited the Institute of Mathematics at Academia Sinica in Taiwan, and he thanks its hospitality.

{\fontsize{10pt}{10.5pt}\selectfont
\bibliographystyle{DavidNatbibStyle}
\bibliography{DavidBibliography}}
\end{document}